\RequirePackage{fix-cm}

\documentclass[smallextended]{svjour3pre}
\smartqed

\usepackage[T1]{fontenc}
\usepackage{hyperref}

\usepackage{bm}
\usepackage{amsmath,amssymb}
\usepackage{graphicx}
\usepackage{mathptmx}
\usepackage{xcolor}
\usepackage{verbatim}
\usepackage[percent]{overpic}
\usepackage{enumerate}

\newcommand{\NR}{\text{W}}
\newcommand{\tol}{{\rm tol}}

\newcommand{\dd}{\mathrm{d}}
\newcommand{\ee}{{\rm e}}
\newcommand{\ii}{{\rm i}}
\newcommand{\Kry}{\mathcal{K}}
\newcommand{\N}{\mathbb{N}}

\newcommand{\R}{\mathbb{R}}
\newcommand{\C}{\mathbb{C}}
\newcommand{\Ono}{\mathcal O}

\DeclareMathOperator{\real}{Re}
\DeclareMathOperator{\imag}{Im}
\DeclareMathOperator{\avg}{avg}
\DeclareMathOperator{\var}{var}

\numberwithin{equation}{section}
\journalname{TBD}

\hypersetup{pdfinfo={
Title={A study of defect-based error estimates for the Krylov approximation of phi-functions},
Author={T. Jawecki},
Subject={Numerical Analysis},
Keywords={matrix exponential, phi-functions, Krylov approximation, a posteriori error estimation, upper bound}
}}

\begin{document}

\title{A study of defect-based error estimates for the Krylov approximation of $\varphi$-functions}

\author{Tobias Jawecki}
\authorrunning{T.\ Jawecki}
\date{\today}
\institute{
Tobias Jawecki \at
Institut f{\"u}r Analysis und Scientific Computing, Technische Universit{\"a}t Wien \\
Wiedner Hauptstrasse 8--10/E101, A-1040 Wien, Austria \\
\email{tobias.jawecki@tuwien.ac.at}
}

\maketitle


\bigskip

\begin{abstract}
Prior recent work, devoted to the study
of polynomial Krylov techniques for the approximation
of the action of the matrix exponential $\ee^{tA}v$,
is extended to the case of associated $\varphi$-functions
(which occur within the class of exponential integrators).
In particular, a~posteriori error bounds and estimates,
based on the notion
of the defect (residual) of the Krylov approximation are considered.
Computable error bounds and estimates are discussed and analyzed.
This includes a new error bound which favorably compares
to existing error bounds in specific cases.
The accuracy of various error bounds is characterized
in relation to corresponding Ritz values of $A$.
Ritz values yield properties of the spectrum of $A$
(specific properties are known a~priori, e.g. for Hermitian or skew-Hermitian matrices)
in relation to the actual starting vector $v$ and can be computed.
This gives theoretical results together with criteria
to quantify the achieved accuracy on the run.
For other existing error estimates the reliability and performance
is studied by similar techniques.
Effects of finite precision (floating point arithmetic) are also taken into account.
\end{abstract}

\section{Introduction}

\paragraph{Overview on prior work.}
The matrix exponential and associated $\varphi$-functions
play a crucial role in some numerical methods for
solving systems of differential equations.
In practice this means that the vector
$ \ee^{tA}v ~~\,\text{for a time step $t$,} $
for a given matrix $A$ and a given vector $v$,
representing the time propagation for a linear initial value problem,
is to be approximated.
Similarly, the associated $\varphi$-functions (see~\eqref{eq.defvarphip} below) conform to
solutions of certain inhomogeneous differential equations.
In particular, evaluation of $\varphi$-functions is used in exponential integrators~\cite{HO10}.

If the matrix $A$ is sparse and large, approximation of the action of these matrix functions
in the class of Krylov subspaces is a general and well-established technique.
For the matrix exponential and $\varphi$-functions
this goes back to early works in the field of chemical physics~\cite{NR83,PL86},
parabolic problems~\cite{GS92},
some nonlinear problems~\cite{FTDR89}, etc.
The case of a symmetric or skew-Hermitian matrix~$ A $ is the most prominent one.
Krylov approximations of the matrix exponential
were early studied for the symmetric case in~\cite{DK89,DK91,Sa92},
and together with $\varphi$-functions in a more general setting~\cite{HLS98,HL97}.

Concerning different approaches for the numerical approximation
of the matrix exponential see~\cite{ML03}.
In~\cite{Sa92} it is shown for the symmetric case that the Krylov approximation
is equivalent to interpolation
of the exponential function at associated Ritz values.
This automatically results in a near-best approximation
among other choices of interpolation nodes,
see also~\cite{DK89,SL96} and future works \cite{BR09}
with similar results for the non-symmetric case.
For other polynomial approaches approximating
the matrix exponential we mention
truncated Taylor series~\cite{MH11} (and many works well in advance),
Chebychev polynomial interpolation~\cite{TK84},
or the Leja method~\cite{CKOR16}.

In general, Krylov approximation (or other polynomial approximations)
result in an accurate approximation
if the time step $t$ in $\ee^{tA}v$ is sufficiently small
or the dimension of the Krylov subspace
(i.e., the degree of the approximating matrix polynomial)
is sufficiently large, 
see for instance~\cite{HL97}.
The dimension of the Krylov subspace is limited in practice,
and large time steps require a restart of the iteration generating the Krylov basis.
A larger time step $t$ can be split into smaller substeps
for which the Krylov approximation can be applied in a nested way.
Such a restart strategy in the sense of a time integrator
was already exploited in~\cite{PL86}.
In particular we refer to the EXPOKIT package~\cite{Si98}.
Similar ideas can be applied for the evaluation of $\varphi$-functions~\cite{HLS98,NW12}.

In practice, a~posteriori error estimates are used
to choose a proper Krylov dimension or proper (adaptive) substeps
if the method is restarted as a time integrator.
Different approaches for a~posteriori error estimation
make use of a series expansion for the error given~\cite{Sa92,Si98}
or use a formulation via the defect (also called residual) of the
Krylov approximation~\cite{DGK98,CM97,BGH13}.
Further a~priori as well as a~posteriori error estimates are given
in~\cite{MN01,Lu08,DMR09,BR09,JL15,WY17,JAK19}.
Restarting via substeps
based on different choices of error estimates is further discussed in~\cite{JAK19}.
A restart with substeps together with a strategy to choose
the Krylov dimension in terms of
computational cost was presented in~\cite{NW12,BK19}.
For various other approaches for restarting
(without adapting the time step)
we refer to~\cite{CM97,EE06,Ta07,Ni07,AEEG08,EEG11,BGH13,Schwe15}.

The influence of round-off errors on the construction of the Krylov basis
in floating point arithmetic was early studied for the symmetric case in~\cite{Pa76,Par98}.
The orthogonalization procedure can behave numerically unstable,
typically due to a loss of orthogonality.
Nevertheless, the near-best approximation property and
related a~priori convergence results are not critically affected~\cite{DK91,DGK98}.
Following~\cite{DGK98}, in the symmetric case the defect obtained in floating point arithmetic
results in numerically stable error estimates.

Beside the polynomial Krylov method, further studies are devoted to
the approximation of matrix functions
using so called extended Krylov subspaces~\cite{DK98,KS10,GG13},
rational Krylov subspaces~\cite{MN04,EH06,Gu10},
or polynomial Krylov subspaces with a harmonic Ritz approach~\cite{HH05,Schwe15,WZX16}.

\paragraph{Overview on results presented here.}
%
In Section~\ref{sec.setting} we introduce the problem setting
and recapitulate basic properties of Krylov subspaces.

In Section~\ref{sec.intreperror} we introduce the {\em defect} associated with
Krylov approximations to $\varphi$-functions,
including the exponential function as the basic case.
Our approach for the defect is different from~\cite{WZX16} and is based on
an inhomogeneous differential equation for the approximation error.
This is used in Theorem~\ref{thm:errorint} to obtain
an integral representation of the error,
also taking effects of floating point arithmetic into account.
\footnote{Cf.~\cite{DGK98} for the case of the matrix exponential.}
In contrast to previous works~(\cite{DGK98,JAK19}), this result 
is extended to $\varphi$-functions here.
Theorem~\ref{thm:errorint} also includes an a~priori upper bound
on the error norm based on an integral of the defect norm.

This upper bound is further analyzed in Section~\ref{sec.uppererrorbound}
to obtain computable a~posteriori bounds,
in particular a new a~posteriori bound (Theorem~\ref{thm.upperbounderrorfull}).
We also study the accuracy of our and other existing defect-based bounds~\cite{JAK19}
with respect to spectral properties of the Krylov Hessenberg matrix
(the representation of $A$ in the orthogonal Krylov basis).
To this end we use properties of divided differences including
a new asymptotic expansion for such given in Appendix~\ref{sec.expansionofrho}.
In Subsection~\ref{sec.quadest} we recapitulate error estimates
based on a quadrature estimate of the defect norm integral,
e.g. the generalized residual estimate~\cite{HLS98}.
We also discuss cases for which the defect norm behaves oscillatory
and reliable quadrature estimates may be difficult to obtain.
In Subsection~\ref{sec.luckybreakdown} we
specify a stopping criterion for the
so-called lucky breakdown in floating point arithmetic
which is justified by our a~posteriori error bounds.

In Section~\ref{sec.numexp} we illustrate our results via numerical experiments.
Here we confine ourselves to the exponential function. Particular application problems
where $\varphi$-functions are used will be documented elsewhere.

\section{Problem statement and Krylov approximation}\label{sec.setting}
We discuss the approximation via Krylov techniques
for evaluation of the matrix exponential,
and in particular of the associated $\varphi$-functions,
for a step size $t>0$ and matrix $A\in\C^{n\times n}$
applied to an initial vector $v\in\C^{n}$. Here,
\begin{equation}\label{eq.defexp}
\ee^{tA} v = \sum_{k=0}^{\infty} \frac{(tA)^k}{k!} v.
\end{equation}
The matrix exponential $u(t) = \ee^{tA} v$
is the solution of the differential equation
$$
u'(t)=Au(t),~~~u(0)=v.
$$
The associated $\varphi$-functions are given by
\begin{equation}\label{eq.defvarphip}
\varphi_p(tA)v = \sum_{k=0}^{\infty} \frac{(tA)^k}{(k+p)!} v,~~~ p \in \N_0.
\end{equation}
This includes the case $\varphi_0 = \exp$.
The matrix functions~\eqref{eq.defexp} and~\eqref{eq.defvarphip}
are defined according to their scalar counterparts.
The following definitions of $\varphi_p$ are equivalent to~\eqref{eq.defvarphip}:
For $z\in\C$ we have $ \varphi_0(z) = \ee^z $, and
\begin{equation}\label{defphiint}
\varphi_p(z) = \frac{1}{(p-1)!} \int_0^1 \ee^{(1-\theta) z} \theta^{p-1}\,\dd\theta,
\quad p \in \N.
\end{equation}
(See also~\cite[Subsection 10.7.4]{Hi08}.)
The function $w_p(t)=t^p\varphi_p(t A) v$ ($ p \in \N $)
is the solution of an inhomogeneous differential equation of the form
\begin{equation}\label{eq.odetphip}
w_p'(t) = A w_p(t) + \frac{t^{p-1}}{(p-1)!} v,~~~w_p(0)=0,
\end{equation}
see for instance~\cite{NW12}. This follows from~\eqref{eq.defvarphip},
$$
\frac{\dd}{\dd t}\big(t^p\varphi_p(t A)v\big)
= \frac{\dd}{\dd t}\Big(\sum_{k=0}^{\infty} \frac{t^{k+p}A^kv}{(k+p)!} \Big)
= A \sum_{k=0}^{\infty} \frac{t^{k+p}A^kv}{(k+p)!} +\frac{t^{p-1}v}{(p-1)!}
= A (t^p\varphi_p(t A)v) + \frac{t^{p-1}v}{(p-1)!}.
$$
The $\varphi$-functions appear for instance
in the field of exponential integrators, see for instance~\cite{HO10}.

For the case of $A$ being a large and sparse matrix,
e.g., the spatial discretization of a partial differential operator using a localized basis,
Krylov subspace techniques are commonly used to approximate~\eqref{eq.defvarphip}
in an efficient way.

\paragraph{Notation and properties of Krylov subspaces.}
\footnote{In the sequel, $ e_j $ denotes the $ j $-th unit vector
in $ \C^m $ or $ \C^n $, respectively.}
We briefly recapitulate the usual notation and properties of standard Krylov subspaces,
see for instance~\cite{Sa03}.
For a given matrix $A\in\C^{n\times n} $, a starting vector $v\in\C^n$
and Krylov dimension $0<m\leq n$, the Krylov subspace is given by
$$
\Kry_m(A,v) = \text{span}(v,Av,\ldots,A^{m-1}v).
$$
Let $ V_m \in \C^{n\times m} $ represent the orthonormal basis of $\Kry_m(A,v)$
with respect to the Hermitian inner product,
constructed by the Arnoldi method and  satisfying $V_m^\ast V_m = I_{m\times m}$.
Its first column is given by $V_m^\ast v = \beta e_1 $ with $\beta = \| v\|_2$.
Here, the matrix
$$
H_m = V_m^\ast A V_m \in \C^{m\times m}
$$
is upper Hessenberg.
We further use the notation $h_{m+1,m}=(H_{m+1})_{m+1,m}\in\R$,
and $v_{m+1}\in\C^{n}$ for the $(m+1)$-th column of $V_{m+1}$,
with $V_m^\ast v_{m+1}=0$ and $\|v_{m+1}\|_2=1$.

The Arnoldi decomposition (in exact arithmetic) can be expressed in matrix form,
\begin{equation} \label{eq.Krylovidexact}
AV_m = V_mH_m + h_{m+1,m} v_{m+1} e_m^\ast .
\end{equation}
\begin{remark}\label{rmk:W(A)W(H)}
The numerical range
$\NR(A)=\{y^\ast A y/y^\ast y,\; 0 \not= y \in\C^n\} $
plays a role in our analysis.
Note that $ \NR(H_m) \subseteq \NR(A) $ (see~\eqref{eq.NRHmissubsetAexactar}).
\end{remark}

\begin{remark}\label{rmk:lucky}
The case $(H_m)_{k+1,k}=0$ occurs if $\Kry_k(A,v)$ is an invariant subspace of $A$,
whence the Krylov approximation given in~\eqref{eq.varphi} below is exact.
This exceptional case is referred to as a {\em lucky breakdown.}
In general we assume that no lucky breakdown occurs,
whence the lower subdiagonal entries of $H_m$ are real and positive,
$0<(H_m)_{j+1,j}$ for $j=1,\ldots,m-1$, and $0<h_{m+1,m}\in\R$.
\end{remark}

For the special case of a Hermitian or skew-Hermitian matrix~$ A $
the Arnoldi iteration simplifies to a three-term recurrence,
the so-called Lanczos iteration.
This case will be addressed in Remark~\ref{rmk:skewHermitiancaseLanczos} below.

\paragraph{Krylov subspaces in floating point arithmetic.}
We proceed with some results for the Arnoldi decomposition in computer arithmetic,
assuming complex floating point arithmetic
with a relative machine precision~$\varepsilon$, see also~\cite{Hi02}.
For practical implementation different variants of the Arnoldi procedure exist,
using different ways for the orthogonalization of the Krylov basis.
These are based on classical Gram-Schmidt, modified Gram-Schmidt,
the Householder algorithm, the Givens algorithm,
or variants of Gram-Schmidt with reorthogonalization
(see also \cite[Algorithm 6.1--6.3]{Sa03} and others).
We refer to~\cite{BLR00} and references therein for an overview
on the stability properties of these different variants.

In the sequel the notation $V_m$, $H_m$, etc.,
will again be used for the result of the Arnoldi method in floating point arithmetic.
We now accordingly adapt some statements formulated in the previous paragraph.
By construction, $H_m$ remains to be upper Hessenberg
with positive lower subdiagonal entries.
Assuming floating point arithmetic we use the notation
$U_m\in\C^{n\times m}$ for a perturbation
of the Arnoldi decomposition~\eqref{eq.Krylovidexact}
caused by round-off, i.e.,
\begin{equation}\label{eq.Krylovidcomputer}
AV_m = V_mH_m + h_{m+1,m} v_{m+1} e_m^\ast + U_m.
\end{equation}
An upper norm bound for $U_m$ was first introduced in~\cite{Pa76} for the Lanczos iteration in real arithmetic.
For different variants of the Arnoldi or Lanczos iteration
this is discussed in~\cite{Ze03} and others.
We assume $\|U_m\|_2$ is bounded by a constant $C_1$ which can depend on $m$ and $n$ in a moderate way
and is sufficiently small in a typical setting,
\begin{subequations}\label{eq.Krylovidcomputerroundoffconst}
\begin{equation}\label{eq.Krylovroundoffconst1}
\|U_m\|_2 \leq C_1 \varepsilon \|A\|_2.
\end{equation}
We further assume that the normalization of the columns of $V_m$ is accurate,
in particular that the $(m+1)$-th basis vector $ v_{m+1} $
is normalized correctly up round-off with a sufficiently small constant $C_2$
(see e.g.~\cite[eq.\,(14)]{Pa76}),
\begin{equation}\label{eq.Krylovroundoffconst2}
|\|v_{m+1}\|_2-1|\leq C_2 \varepsilon.
\end{equation}
Concerning $V_{m+1}$ which represents an orthogonal basis in exact arithmetic,
numerical loss of orthogonality has been well-studied.
Loss of orthogonality can be significant (see for instance \cite{Par98,BLR00} and others), depending on the starting vector $ v $.
Reorthogonalization schemes or orthogonalization via Householder or Givens algorithm
can be used to obtain orthogonality of $V_{m+1}$ on a sufficiently accurate level.

The numerical range of $H_m$ obtained in floating point arithmetic
(see~\eqref{eq.Krylovidcomputer}) can be characterized as
\begin{equation}\label{eq.Krylovnumrangecomputer}
\NR(H_m) \subseteq U_{C_3\varepsilon}(\NR(A)),
\end{equation}
\end{subequations}
with $U_{C_3\varepsilon}(\NR(A))$ being the neighborhood of $\NR(A)$ in $\C$ with a distance $C_3\varepsilon$.
With the assumption that $V_{m+1}$ is sufficiently close to orthogonal
(e.g., semiorthogonal~\cite{Si84}),
the constant $C_3$ in~\eqref{eq.Krylovnumrangecomputer}
(which also depends on $C_1$ and problem sizes) can be shown to be moderate-sized.
Further details on this aspect are given in Appendix~\ref{sec.appendixAKrylov}.

\paragraph{Krylov approximation of $\varphi$-functions.}

\footnote{Remark concerning notation: '$u$' objects live in $ \C^n $,
          and '$y$' objects live in $ \C^m $.}
Let $V_m\in\C^{n\times m}$, $H_m\in\C^{m\times m}$ and $\beta\in\R$
be the result of the Arnoldi method in floating point arithmetic for $\Kry_m(A,v)$ as described above.
For a time-step $0<t\in\R$ and $p \geq 0$ the vector $ \varphi_p(tA)v $
can be approximated in the Krylov subspace $\Kry_m(A,v)$
by the {\em Krylov propagator}
\begin{subequations}
\begin{equation}\label{eq.Krylovprpagator}
u_{p,m}(t):= V_m \varphi_p(tV_m^\ast AV_m)V_m^\ast v = \beta V_m \varphi_p(tH_m) e_1 ,~~~p\in\N.
\end{equation}
The special case $p=0$ reads
\begin{equation}\label{eq.Krylovprpagatorexp}
u_{0,m}(t) = \beta V_m \ee^{tH_m} e_1.
\end{equation}
\end{subequations}
We remark that the small-dimensional problem $\varphi_p(tH_m)e_1\in\C^m$,
typically with $m\ll n$, can be evaluated cheaply by standard methods.
In the sequel we denote
\begin{equation}\label{eq.varphi}
y_{p,m}(t)=\beta \varphi_p(tH_m) e_1\in\C^m,
\quad \text{i.e.,} \quad
u_{p,m}(t) = V_m y_{p,m}(t).
\end{equation}
For $p=0$ the small dimensional problem $y_{0,m}(t) = \beta \ee^{t H_m} e_1$ solves the differential equation
\begin{equation}\label{eq.odeypzero}
{y}'_{0,m}(t) = H_m {y}_{0,m}(t),~~~{y}_{0,m}(0)=\beta e_1,
\end{equation}
For later use we introduce the notation
\begin{subequations}
\begin{equation}\label{eq.yhatistpyp}
\widehat{y}_{p,m}(t) = t^p y_{p,m}(t),
\end{equation}
which for $p\in\N$ and according to~\eqref{eq.odetphip} satisfies the differential equation
\begin{equation}\label{eq.odetphipy}
\widehat{y}'_{p,m}(t) = H_m \widehat{y}_{p,m}(t) + \tfrac{t^{p-1}}{(p-1)!} \beta e_1,~~~\widehat{y}_{p,m}(0)=0,
\end{equation}
\end{subequations}

\begin{remark}\label{rmk:roundingeffects}
Although we take rounding effects in the Arnoldi decomposition into account,
we do not give a full study of round-off errors at this point.
Round-off errors in substeps such as the evaluation of
$y_{p,m}(t)$
or the matrix-vector multiplication $V_m y_{p,m}(t)$ will be ignored.
We refer to~\cite{Hi02} for a more general study of these effects.
\end{remark}

\begin{remark}\label{rmk:skewHermitiancaseLanczos}
In the special cases $A=B$ or $A=\ii B$ for a Hermitian matrix $B\in\C^{n\times n}$
(with $A$ being skew-Hermitian in the latter case)
the orthogonalization of the Krylov basis of $\Kry_m(B,v)$ simplifies to
a three-term recursion, the so-called Lanczos method.
In the skew-Hermitian case ($A=\ii B$) the Krylov propagator~\eqref{eq.Krylovprpagator}
can be evaluated by $\beta V_m \varphi_p(\ii tH_m) e_1$, i.e.,
we approximate the function $\lambda\mapsto \varphi_p(\ii t \lambda)$ in the Krylov subspace $\Kry_m(B,v)$.
The advantage is a cheaper computation of the Krylov subspace
in terms of computational cost
and better conservation of geometric properties.
For details we refer to the notation $\ee^{\sigma t B}$ as introduced in~\cite{JAK19},
with $\sigma=\pm\ii$ and a Hermitian matrix $B$ for the skew-Hermitian case.
\end{remark}

\paragraph{The error of the Krylov propagator.}

We denote the error of the Krylov propagator given in~\eqref{eq.varphi} by
\begin{equation}\label{def.errorpm}
l_{p,m}(t)
= \beta V_m\varphi_p(tH_m)e_1 - \varphi_p(tA)v,~~~p\in\N_0.
\end{equation}
We are further interested in computable a~posteriori estimates for the error norm $\zeta_{p,m}(t) \approx \| l_{p,m}(t) \|_2$,
which in the best case can be proven to be upper bounds on the error norm $ \| l_{p,m}(t) \|_2 \leq \zeta_{p,m}(t) $.
Norm estimates of the error~\eqref{def.errorpm} can be used in practice to stop the Krylov iteration after $k$ steps if $\|l_{p,k}(t)\|_2$
satisfies~\eqref{def.erroresttol} below,
or to restrict the time-step $t$ to obtain an accurate approximation and restart the method with the remaining time.
For details on the total error with this restarting approach see also~\cite{Si98,JAK19}.

A prominent task is to test if the error norm per unit step is bounded by a tolerance $\tol$,
\begin{equation}\label{def.erroresttol}
\zeta_{p,m}(t) \leq t \cdot \tol,~~~\text{which should entail}~~~\|l_{p,m}(t)\|_2 \leq t \cdot \tol.
\end{equation}
In case of $\zeta_{p,m}(t)$ being an upper bound on the error norm,
this results in a reliable bound on the error norm~\eqref{def.erroresttol}.

\section{An integral representation for the error of the Krylov propagator}\label{sec.intreperror}

We proceed with discussing the error $ l_{p,m} $ of the Krlyov propagator.
To this end we first define its scalar \textit{defect} by
\begin{subequations}
\begin{equation}\label{eq.defpm}
\delta_{p,m}(t) =
\beta e_m^\ast t^p \varphi_p(t H_m) e_1 = t^p \big(y_{p,m}(t)\big)_m\in\C,
\end{equation}
and the \textit{defect integral} by\footnote{
This and the result of Theorem~\ref{thm:errorint} remain valid for the case $t=0$.}
\begin{equation}\label{eq.defint}
L_{p,m}(t) = \frac{h_{m+1,m}}{t^p}\int_0^t |\delta_{p,m}(s)| \,\dd s\in\R.
\end{equation}
\end{subequations}
%
%
%
\begin{theorem}\label{thm:errorint}
Let $\delta_{p,m}(t)\in\C$ be the defect defined in~\eqref{eq.defpm}.
For $y_{p,m}(t)\in\C^m$ defined in~\eqref{eq.varphi}
and a numerical perturbation $U_m\in\C^{n\times m}$
of the Arnoldi decomposition (see~\eqref{eq.Krylovidcomputer}), we have:
\begin{enumerate}[(a)]
\item
The error $l_{p,m}(t)$ of the Krylov propagator (see~\eqref{def.errorpm})
enjoys the integral representation
\begin{subequations}
\begin{equation}\label{eq.thm1deriv01-p}
l_{p,m}(t) = -\frac{h_{m+1,m}}{t^p}\int_0^t \ee^{(t-s)A}v_{m+1} \delta_{p,m}(s) \,\dd s
- \frac{1}{t^p}\int_0^t \ee^{(t-s)A} U_m s^p y_{p,m}(s) \,\dd s.
\end{equation}
\item
For given machine precision $\varepsilon$ and
constants $C_1$, $C_2$ representing round-off effects
(see~\eqref{eq.Krylovroundoffconst1},\eqref{eq.Krylovroundoffconst2}),
and with $\kappa_1 = \max_{s\in[0,t]}\|\ee^{sA}\|_2 $ and $\kappa_2 = \max_{s\in[0,t]}\|\ee^{sH_m}\|_2 $
the error norm is bounded by
\begin{equation}\label{eq.errnormboundint}
\|l_{p,m}(t)\|_2 \leq (1+C_2 \varepsilon )  \kappa_1 L_{p,m}(t) + C_1\varepsilon \|A\|_2 \frac{\beta \kappa_1\kappa_2 t}{(p+1)!},
\end{equation}
with the defect integral $L_{p,m}(t)$ defined in~\eqref{eq.defint}.
\end{subequations}
\end{enumerate}
\end{theorem}
\begin{proof}
$~$
\begin{enumerate}[(a)]
\item
For the exact matrix function we use the notation
$$
u_p(t)=\varphi_p(tA)v, \quad \text{and} \quad  w_p(t)=t^p u_p(t).
$$
For the Krylov propagator we denote
$$
u_{p,m}(t)=V_m y_{p,m}(t) ~~~\text{with}~~ y_{p,m}(t)=\beta \varphi_p(tH_m)e_1
$$
(see~\eqref{eq.varphi}),
and we also define
$$w_{p,m}(t) = t^p u_{p,m}(t) = V_m \widehat{y}_{p,m}(t),
~~~\text{with}~~ \widehat{y}_{p,m}(t)=t^p{y}_{p,m}(t)~\;\text{defined in~\eqref{eq.yhatistpyp}.}
$$
\begin{itemize}
\item For $ p \in \N $, the functions
$w_{p}(t)$ and $w_{p,m}(t)$ satisfy the differential equations
(see~\eqref{eq.odetphip}, \eqref{eq.odetphipy})
\begin{equation}\label{eq.thm1deriv01}
\begin{aligned}
&w'_{p,m}(t)  = V_m \widehat{y}'_{p,m}(t)  = V_m\big( H_m \widehat{y}_{p,m}(t) +  \tfrac{t^{p-1}}{(p-1)!} \beta e_1\big),\\
&w'_{p}(t) = A w_{p}(t) + \tfrac{t^{p-1}}{(p-1)!} v, \quad
\text{and}~~~{w}_{p}(0)={w}_{p,m}(0)=0. 
\end{aligned}
\end{equation}
\item For $p=0$, i.e., $w_0(t)=u_0(t)$ and $w_{0,m}(t)=V_my_{0,m}(t)$,
      according to~\eqref{eq.odeypzero} we have
$$
\begin{aligned}
&w'_{0}(t) = A w_{0}(t), \quad
w'_{0,m}(t)  =  V_m H_m {y}_{0,m}(t), \\
&\text{and}~~~w_{0}(0)=v,~~~w_{0,m}(0)=\beta V_m e_1 =v.
\end{aligned}
$$
\end{itemize}
\paragraph{Local error representation in terms of the defect.}
We defined the scaled error
$$
\widehat{l}_{p,m}(t) = w_{p,m}(t) - w_p(t) = t^p l_{p,m}(t).
$$
\begin{itemize}
\item For $ p \in \N $ the scaled error satisfies 
\begin{equation}\label{eq.odeerror}
\widehat{l}'_{p,m}(t)
= w'_{p,m}(t) - w'_p(t)
= A\,\widehat{l}_{p,m}(t) + d_{p,m}(t),~~~\widehat{l}_{p,m}(0)=0,
\end{equation}
with the {\em defect} of $ w_{p,m}(t) $ with respect to the differential equation~\eqref{eq.thm1deriv01},
\begin{align*}
d_{p,m}(t)
&= w'_{p,m}(t) - A w_{p,m}(t)  - \tfrac{t^{p-1}}{(p-1)!}  \\
&= V_m\big( H_m \widehat{y}_{p,m}(t) +  \tfrac{t^{p-1}}{(p-1)!} \beta e_1\big) - A V_m \widehat{y}_{p,m}(t) - \tfrac{t^{p-1}}{(p-1)!} v \\
&= \big( V_m H_m - A V_m \big) \widehat{y}_{p,m}(t) + \tfrac{t^{p-1}}{(p-1)!}(\beta V_m e_1 - v).
\end{align*}
Together with~\eqref{eq.Krylovidcomputer} and using of $\beta V_me_1=v$ the defect can be written as
\begin{equation*} 
d_{p,m}(t) = - h_{m+1,m} (e_m^\ast \widehat{y}_{p,m}(t)) v_{m+1} - U_m \widehat{y}_{p,m}(t). 
\end{equation*}
\item For $p=0$, in an analogous way we obtain
\begin{equation*} 
d_{0,m}(t) = - h_{m+1,m} (e_m^\ast {y}_{0,m}(t))  v_{m+1} - U_m {y}_{0,m}(t).
\end{equation*}
\end{itemize}
We conclude
\begin{equation}\label{eq.defvecpall}
d_{p,m}(t) = -h_{m+1,m} \delta_{p,m}(t) v_{m+1} -  t^p U_m y_{p,m}(t),~~~p\in\N_0,
\end{equation}
with the scalar defect defined in~\eqref{eq.defpm}.
Due to~\eqref{eq.odeerror} we have
\begin{equation*}
\widehat{l}_{p,m}(t) = \int_0^t \ee^{(t-s) A} d_{p,m}(s) \dd s,~~~p\in\N_0,
\end{equation*}
and for ${l}_{p,m}(t) = t^{-p} \widehat{l}_{p,m}(t)$ together with~\eqref{eq.defvecpall}
this implies~\eqref{eq.thm1deriv01-p}.
\item
With $\kappa_1 =\max_{t\in[0,t]} \|\ee^{tA}\|_2$,
$\|U_m\|_2\leq C_1\varepsilon \|A\|_2$ and $\|v_{m+1}\|_2\leq 1+C_2\varepsilon$,
the representation~\eqref{eq.thm1deriv01-p} implies the upper bound
\begin{equation}\label{eq.interrestnormbound1}
\|l_{p,m}(t)\|_2 \leq (1+C_2\varepsilon) \kappa_1 \frac{ h_{m+1,m} }{t^p} \int_0^t |\delta_{p,m}(s)|  \, \dd s
+ C_1\varepsilon \|A\|_2 \frac{\kappa_1}{t^p}   \int_0^t s^p \| y_{p,m}(s)\|_2 \, \dd s .
\end{equation}
With the defect integral $L_{p,m}(t)$ defined in~\eqref{eq.defint}
we obtain the first term in~\eqref{eq.errnormboundint}.
For the second integral term (with $ y_{p,m}(t) = \beta \varphi_p(tH_m)e_1 $)
we use the upper bound
\begin{equation}\label{eq.interrestnormbound2}
\int_0^t s^p \| \varphi_p(sH_m)e_1\|_2 \, \dd s \leq  \max_{s\in[0,t]}\| \varphi_p(sH_m)e_1\|_2 \frac{t^{p+1}}{p+1}.
\end{equation}
\begin{itemize}
\item
For $p \in \N$ we apply the integral representation due to~\eqref{defphiint} for $\varphi_p(t H_m) e_1$
to obtain the norm bound
\begin{equation}\label{eq.interrestnormbound3}
\max_{s\in[0,t]} \|\varphi_p(s H_m) e_1\|_2
\leq \frac{\max_{s\in[0,t]} \|\ee^{s H_m}\|_2}{(p-1)!} \int_0^1\theta^{p-1}\,\dd\theta
= \frac{\max_{s\in[0,t]} \|\ee^{s H_m}\|_2 }{p!}.
\end{equation}
\item
For $p=0$ we obtain~\eqref{eq.interrestnormbound3} in a direct way.
\end{itemize}
Combining~\eqref{eq.interrestnormbound2} with~\eqref{eq.interrestnormbound3}
and denoting $\kappa_2 = \max_{s\in[0,t]}\|\ee^{sH_m}\|_2$ we obtain
\begin{equation*}
\frac{\kappa_1}{t^p}   \int_0^t s^p \| y_{p,m}(s)\|_2 \, \dd s  \leq \frac{\beta \kappa_1\kappa_2 t}{(p+1)!}.
\end{equation*}
Combining these estimates with~\eqref{eq.interrestnormbound1}
we conclude~\eqref{eq.errnormboundint}. \qed
\end{enumerate}
\end{proof}

\begin{remark}
The error norm of the Krylov propagator scales with $\kappa_1 = \max_{s\in[0,t]}\|\ee^{sA}\|_2 $ and $\kappa_2 = \max_{s\in[0,t]}\|\ee^{sH_m}\|_2 $
in a natural way.
\footnote{Taking the maximum $\max_{s\in[0,t]}$ in the definition
of $\kappa_1$ and $\kappa_2$ is necessary to cover the case $p>0$.
For the special case $p=0$ the upper norm bound
given in Theorem~\ref{thm:errorint} can be adapted
to scale with $\ee^{t\mu_2(A)}$}.
It is well known that
$$
\|\ee^{tA}\|_2 \leq \ee^{t \mu_2(A)}~~\text{with the logarithmic norm}~\,
\mu_2(A) = \max\{\real(\NR(A))\} = \max \{ \text{spec}(A+A^\ast)/2\},
$$
see for instance~\cite[Theorem 10.11]{Hi08}.
Problems with $\mu_2(A)>0$ can be arbitrary ill-conditioned and difficult to solve with proper accuracy.
(For further results on the stability of the matrix exponential see also~\cite{ML03,Lo77}.)
We will not further discuss problems with $\mu_2(A)>0$ and assume $\mu_2(A)\leq 0$.
We refer to the case $\mu_2(A) \leq 0$ as the dissipative case, with $\kappa_1=1$.
\end{remark}

For the dissipative case with $\mu_2(A)\leq 0$
the error bound~\eqref{eq.errnormboundint} from Theorem~\ref{thm:errorint} reads
\begin{equation}\label{eq.definterrorfull}
\|l_{p,m}(t)\|_2 \leq (1+C_2 \varepsilon ) L_{p,m}(t)
+ C_1\varepsilon \|A\|_2 \frac{\beta \kappa_2 t}{(p+1)!}.
\end{equation}
The dissipative behavior of $\ee^{tA}$ carries over to the Krylov propagator
up to a perturbation which depends on round-off errors, including the loss of orthogonality of $V_m$.
In terms of the numerical range $\NR(H_m)$,
with $\NR(H_m)\subseteq U_{C_3\varepsilon}(\NR(A))$ we have $\mu_2(H_m)\leq \mu_2(A) + C_3\varepsilon$,
for a constant $C_3\varepsilon$ depending on round-off effects~\eqref{eq.Krylovnumrangecomputer}.
Thus, $\mu_2(H_m) \leq C_3 \varepsilon$ and $\kappa_2 \leq \ee^{t C_3 \varepsilon}$.

Our aim is to construct an upper norm bound for the error per unit step~\eqref{def.erroresttol}
via~\eqref{eq.definterrorfull}.
Let the tolerance $\tol$ be given and $t$ be a respective time step for~\eqref{def.erroresttol}.
Then the round-off error terms in~\eqref{eq.definterrorfull} are negligible if
\begin{equation}\label{eq.definterrorcondition}
C_2 \varepsilon\ll 1,~~~\text{and}~~ C_1 \varepsilon \|A\|_2 \beta \ee^{t C_3 \varepsilon} /(p+1)! \ll \tol.
\end{equation}
Concerning the constants $C_1$, $C_2$ and $C_3$ see~\eqref{eq.Krylovidcomputerroundoffconst}.
We recapitulate that $C_1$ and $C_2$ given in~\eqref{eq.Krylovroundoffconst1} and~\eqref{eq.Krylovroundoffconst2}
can be considered to be small enough in a standard Krylov setting.
The constant $C_3$ can be larger in the case of a loss of orthogonality
of the Krylov subspace, which can however be avoided at the cost of additional computational effort.
The constant $C_3$ only appears as an exponential prefactor
for the round-off term in~\eqref{eq.definterrorcondition}
and is less critical compared to $C_1$ and $C_2$.

With the previous observation on the round-off errors taken into account in~\eqref{eq.definterrorfull}
we consider the following upper bound to be stable in computer arithmetic
in accordance to a proper value of $\tol$, see~\eqref{eq.definterrorcondition}.

\begin{corollary}\label{thm:errorintexactarithmetic}
For the case $\mu_2(A)\leq 0$ and
with the assumption that round-off error is negligible,
the error of the Krylov propagator
is bounded by the defect integral $L_{p,m}(t)$,
$$
 \|l_{p,m}(t)\|_2 \leq \frac{h_{m+1,m}}{t^p}\int_0^t |\delta_{p,m}(s)| \,\dd s = L_{p,m}(t),~~~p\in\N_0.
$$
\end{corollary}
Note that the defect norm $|\delta_{p,m}(s)|$ cannot be integrated exactly
in general. This point will further be studied in the sequel.

\paragraph{Representing the defect in terms of divided differences.}
Divided differences play an essential role in this work.
We use the notation
$$
f[\lambda_1,\ldots,\lambda_m]
$$
for the divided differences of a function $f$
over the nodes $\lambda_1,\ldots,\lambda_m$.
(This is to be understood in the confluent sense
for the case of multiple nodes $ \lambda_j $, see for instance~\cite[Section B.16]{Hi08}.)

%
%
%

\begin{theorem}[see for instance~\cite{CM97}]\label{thm.matfcttodd}
Let $H_m\in\C^{m\times m}$ be an upper Hessenberg matrix with positive secondary diagonal entries,
$0<(H_m)_{j+1,j}\in\R$ for $j=1,\ldots,m-1$, and eigenvalues $\lambda_1,\ldots,\lambda_m$.
Let $f$ be an analytic function for which $f(H_m)$ is well defined.
Then,
$$
e_m^\ast f(H_m)e_1 = \gamma_m f[\lambda_1,\ldots,\lambda_m],
$$
with $\gamma_m=\prod_{j=1}^{m-1} (H_m)_{j+1,j}$.
\end{theorem}

%
%
%
For $f = (\varphi_p)_t:\lambda\mapsto \varphi_p(t\lambda)$
we will also make use of the following result.
\footnote{Theorem~\ref{thm.emphie1toemexpe1} can be generalized to the case
$ t^p e_m^\ast \varphi_{k+p}(t H_m)e_1 = e_{m+p}^\ast \varphi_k(t\widetilde{H}_{p,m}) e_1$
with $ k \in \N $,
see~\cite[Theorem 2.1]{MH11}. The case $k=0$ is sufficient for our purpose.}

\begin{theorem}[Corollary 1 in~\cite{Si98}]\label{thm.emphie1toemexpe1}
(Expressing $ \varphi $-functions via dilated $ \exp $-functions.)
For $ t \in \R $,
\begin{equation*}
t^p e_m^\ast \varphi_p(t H_m)e_1 = e_{m+p}^\ast \exp(t\widetilde{H}_{p,m}) e_1
\end{equation*}
with
$$
\widetilde{H}_{p,m} =
\begin{pmatrix}
H_m & 0_{m\times p}\\
e_1 e_m^\ast & J_{p\times p}
\end{pmatrix}
\in\C^{(m+p)\times(m+p)}~~~\text{and}~~~
 J_{p\times p}
=
\begin{pmatrix}
0&  &&\\
1 & 0 &&\\
  & \ddots &\ddots& \\
 && 1&0
\end{pmatrix}\in\C^{p\times p}.
$$
\end{theorem}

The matrix $\widetilde{H}_{p,m}$ in Theorem~\ref{thm.emphie1toemexpe1}
is block triangular with eigenvalues equal to those of $H_m$ and $J_{p\times p}$.
Therefore, $ \text{spec}(\widetilde{H}_m) = \{\lambda_1,\ldots,\lambda_m,0,\ldots,0\}$,
with $0$ as an eigenvalue of multiplicity $p$ (at least).
In our context,
$\widetilde{H}_m$ is upper Hessenberg with a positive lower secondary diagonal and
$\gamma_m=\prod_{j=1}^{m-1} (H_m)_{j+1,j} = \prod_{j=1}^{m+p-1} (\widetilde{H}_m)_{j+1,j}$.
In accordance with Theorem~\ref{thm.matfcttodd} the result of Theorem~\ref{thm.emphie1toemexpe1} holds for divided differences in a similar manner,
$$
t^p (\varphi_p)_{t}[\lambda_1,\ldots,\lambda_m]
= \exp_{t} [ \lambda_1,\ldots,\lambda_m,\underbrace{0,\ldots,0}_{\text{$p$ times}} ].
$$

With Theorem~\ref{thm.matfcttodd} and~\ref{thm.emphie1toemexpe1}
the following equivalent formulations can be used the rewrite the scalar defect
$\delta_{p,m}(t)$ defined in~\eqref{eq.defpm}.
\begin{corollary}\label{thm.defectequal}
Let $\delta_{p,m}(t)$ be the scalar defect given in~\eqref{eq.defpm} for the upper Hessenberg matrix $H_m\in\C^{m\times m}$
with positive secondary diagonal entries.
Denote $0<\gamma_m = \prod_{j=1}^{m-1} (H_m)_{j+1,j}$.
Let $\widetilde{H}_{p,m}\in\C^{m+p}$ be given as in Theorem~\ref{thm.emphie1toemexpe1}.
For the scalar defect we obtain the following equivalent formulations:
\begin{enumerate}[(i)]
\item $\delta_{p,m}(t) = \beta e_m^\ast t^p \varphi_p(t H_m) e_1$
\item \qquad\quad $ =\beta \gamma_m  t^p (\varphi_p)_t[\lambda_1,\ldots,\lambda_m]$
\item\label{thm.defectequal.extendH} \qquad\quad $ = \beta e_{m+p}^\ast \exp(t \widetilde{H}_{p,m}) e_1$
\item \qquad\quad $ = \beta \gamma_m \exp_t[\lambda_1,\ldots,\lambda_m,0_p]$\footnote{
Here we introduce the notation
$(\lambda_1,\ldots,\lambda_m,0_p)=(\lambda_1,\ldots,\lambda_m,0,\ldots,0)\in\C^{m+p}$ for $p\in\N_0$.}
\end{enumerate}
\end{corollary}

We remark that the eigenvalues $\lambda_1,\ldots,\lambda_m$
of the Krylov Hessenberg matrix $H_m$ are
also referred to as Ritz values (of $A$) in the literature.


\section{Computable a~posteriori error bounds for the Krylov propagator}\label{sec.uppererrorbound}

The following two propositions are used for the proof of
Theorem~\ref{thm.upperbounderrorfull} below.\footnote{
               We use the notation introduced in the previous sections.}
\begin{proposition}\label{thm.intphiA}
For arbitrary nodes $\lambda_j\in\C$ and $p\in\N_0$,
$$
\int_0^t s^p (\varphi_p)_s[\lambda_1,\ldots,\lambda_k] \,\dd s = t^{p+1} (\varphi_{p+1})_t[\lambda_1,\ldots,\lambda_k].
$$
\end{proposition}
\begin{proof}
See Appendix~\ref{sec.propertiesofdd}.
\end{proof}

\begin{proposition}[Lemma including eq.~(5.1.1) in~\cite{MNP84}]\label{thm.realnodesposphiandupperbound}
For arbitrary nodes $\lambda_j=\xi_j+\ii\eta_j \in \C$,
$$
|\exp_t[\lambda_1,\ldots,\lambda_k]| \leq \exp_t[\xi_1,\ldots,\xi_k].
$$
\end{proposition}
\begin{proof}
See Appendix~\ref{sec.propertiesofdd}.
\end{proof}

We now derive upper bounds for the error
via its representation by the defect integral~\eqref{eq.defint}.
%
%
%
\begin{theorem}\label{thm.upperbounderrorfull}
Let $p\in\N_0$, $\mu_2(A)\leq 0$, and assume that round-off errors are sufficiently small
(see Corollary~\ref{thm:errorintexactarithmetic}).
For the eigenvalues of $H_m$ we write $\lambda_j=\xi_j+\ii\eta_j$, $j=1,\ldots,m$.
An upper bound on the error norm is given by
\begin{equation}\label{eq.thm1realnodes}
\|l_{p,m}(t)\|_2 \leq  \beta h_{m+1,m} \gamma_m t (\varphi_{p+1})_t[\xi_1,\ldots,\xi_m].
\end{equation}
\end{theorem}
\begin{proof}
Due to Corollary~\ref{thm.defectequal}, (iv),
\begin{subequations}
\begin{equation}\label{eq.thm1intmftodd}
\delta_{p,m}(t) = \beta \gamma_m  \exp_t [\lambda_1,\ldots,\lambda_m,0_{p}].
\end{equation}
The divided differences in~\eqref{eq.thm1intmftodd}
span over complex nodes $\lambda_1,\ldots,\lambda_m $ and $ 0_{p}\in\C^p$,
with real parts $\xi_1,\ldots,\xi_m$.
Propositions~\ref{thm.realnodesposphiandupperbound} and~\ref{thm.intphiA} imply
$$
\int_0^t |\exp_s [\lambda_1,\ldots,\lambda_m,0_{p}]| \,\dd s
\leq \int_0^t \exp_s[\xi_1,\ldots,\xi_m,0_{p}] \,\dd s
= t (\varphi_{1})_t [\xi_1,\ldots,\xi_m,0_{p}].
$$
From Corollary~\ref{thm.defectequal} we obtain
\begin{equation}\label{eq.thm1intddback}
t (\varphi_{1})_t[\xi_1,\ldots,\xi_m,0_{p}]
= \exp_t[\xi_1,\ldots,\xi_m,0_{p+1}] = t^{p+1} (\varphi_{p+1})_t[\xi_1,\ldots,\xi_m].
\end{equation}
\end{subequations}
Eqs.~\eqref{eq.thm1intmftodd}--\eqref{eq.thm1intddback}
together with Corollary~\ref{thm:errorintexactarithmetic} imply~\eqref{eq.thm1realnodes}.
\qed
\end{proof}

For the case of $H_m$ having real eigenvalues,
the assertion of Theorem~\ref{thm.upperbounderrorfull} can be reformulated
in the following way (see~\cite[Proposition 6]{JAK19}).
%
%
%
\begin{corollary}\label{cor.uppererrorboundreal}
Assume $\mu_2(A)\leq 0$
and that round-off errors are sufficiently small
(see Corollary~\ref{thm:errorintexactarithmetic}).
For the case of $H_m$ having real eigenvalues $\lambda_1,\ldots,\lambda_m\in\R $, the upper bound on the error norm in Theorem~\ref{thm.upperbounderrorfull}
yields an exact evaluation of the defect integral. Hence,
$$
\|l_{p,m}(t)\|_2 \leq  L_{p,m}(t) =
\beta h_{m+1,m} t \big( e_m^\ast \varphi_{p+1}(tH_m) e_1 \big).
$$
\end{corollary}

As a further corollary we formulate an upper bound on the error norm
which is cheaper to evaluate compared
to the bound from Theorem~\ref{thm.upperbounderrorfull}
but may be less tight.
%
%
%
Using the Mean Value Theorem,~\cite[eq. (B.26)]{Hi08} or~\cite[eq. (44)]{Bo05},
for the divided differences in Theorem~\ref{thm.upperbounderrorfull},
eq.~\eqref{eq.thm1realnodes} we obtain the following result
which corresponds to~\cite[Theorem 1 and 2]{JAK19}.
For the exponential of a skew-Hermitian matrix
a similar error estimate has been used in~\cite{KBC05}
and is based on ideas of~\cite{PL86} with some lack of theory.
\begin{corollary}\label{thm.upperbounderrorfull2}
Let $p\in\N_0$, $\mu_2(A)\leq 0$,
and assume that round-off errors are sufficiently small
(see Corollary~\ref{thm:errorintexactarithmetic}).
Let ${\xi}_{\max} = 0 $ for $p\in\N$ and ${\xi}_{\max} = \max_{j=1,\ldots,m} \xi_j \leq 0$ for $p=0$
and eigenvalues $\lambda_j=\xi_j+\ii\mu_j\in\C$ of $H_m$.
An upper bound on the error norm is given by
$$
\|l_{p,m}(t)\|_2 \leq  \beta h_{m+1,m} \frac{ \gamma_m t^m \ee^{t{\xi}_{\max}} }{(m+p)!}
\leq \beta h_{m+1,m} \frac{ \gamma_m t^m }{(m+p)!}.
$$
\end{corollary}

For the case of $H_m$ having purely imaginary eigenvalues,
the divided differences in Theorem~\ref{thm.upperbounderrorfull}
(see~\eqref{eq.thm1realnodes}) can be evaluated directly via~\cite[eq. (B.27)]{Hi08},
$$
t (\varphi_{p+1})_t[0_m] = t^{-p} \exp_t[0_{m+p+1}]   =   \frac{ t^m }{(m+p)!},
$$
hence the assertions of Theorem~\ref{thm.upperbounderrorfull} and Corollary~\ref{thm.upperbounderrorfull2}
coincide in this case.

\paragraph{Accuracy of the previously specified upper bounds on the error norm.}

In the following we again denote $\lambda_1,\ldots,\lambda_m\in\C$
for the eigenvalues of $H_m$,
with $\lambda_j=\xi_j+\ii\eta_j$.
For the scalar defect $\delta_{p,m}(t)$ (see~\eqref{eq.defpm})
we recapitulate Corollary~\ref{thm.defectequal}, in particular
\begin{equation}\label{eq.ddboundofdeffrom00}
\delta_{p,m}(t) = \beta \gamma_m t^p (\varphi_p)_t[\lambda_1,\ldots,\lambda_m]
=  \beta \gamma_m \exp_t[\lambda_1,\ldots,\lambda_m,0_p].
\end{equation}
Theorem~\ref{thm.upperbounderrorfull} and its corollaries make use of the error bound given in Corollary~\ref{thm:errorintexactarithmetic}
and computable upper bounds on the defect integral $L_{p,m}(t)$.
A refinement of the upper bound from Corollary~\ref{thm:errorintexactarithmetic} would require further applications of the large-dimensional
matrix-vector product with $A\in\C^{n\times n}$
and has been shown to be inefficient in terms of computational cost,
see also~\cite[Remark 7]{JAK19}.
The computable upper bounds on the defect integral $L_{p,m}(t)$ will be further discussed.
We recapitulate the upper bound of the divided differences given in Proposition~\ref{thm.realnodesposphiandupperbound},
\begin{equation}\label{eq.ddboundofdeffromabove.recap}
|\exp_t[\lambda_1,\ldots,\lambda_m,0_p]|\leq \exp_t[\xi_1,\ldots,\xi_m,0_p].
\end{equation}
Thus, in the case of $H_m$ having eigenvalues with a sufficiently small imaginary part,
the upper bound in Proposition~\ref{thm.realnodesposphiandupperbound},
is tight.
In the following proposition this statement is made more precise.
\begin{proposition}[Part of a proof in~\cite{MNP84}, eq. (5.2.3)]\label{thm.ddlamlowerbound}
For nodes $\lambda_j=\xi_j+\ii\eta_j \in \C$ and $t\geq 0$
with $\max_j t|\eta_j| \leq \widetilde{\eta}_t < \pi/2$,
$$
0<\cos(\widetilde{\eta}_t) \exp_t[\xi_1,\ldots,\xi_k]  \leq |\exp_t[\lambda_1,\ldots,\lambda_k]|.
$$
\end{proposition}
\begin{proof}
See Appendix~\ref{sec.propertiesofdd}.
\end{proof}
Under the assumptions of Proposition~\ref{thm.ddlamlowerbound} we conclude
\begin{equation}\label{eq.cosexp01}
0<\cos(\widetilde{\eta}_t) \exp_t[\xi_1,\ldots,\xi_m,0_p]  \leq |\exp_t[\lambda_1,\ldots,\lambda_m,0_p]|.
\end{equation}
With~\eqref{eq.ddboundofdeffrom00},~\eqref{eq.ddboundofdeffromabove.recap},~\eqref{eq.cosexp01}
and following the proof of Theorem~\ref{thm.upperbounderrorfull}
the defect integral in~\eqref{eq.defint} can be enclosed by
\begin{equation}\label{eq.ddboundofdefintfrom0}
0<\cos(\widetilde{\eta}_t) \cdot \beta \gamma_m h_{m+1,m} t (\varphi_{p+1})_t[\xi_1,\ldots,\xi_m]
\leq L_{p,m}(t)
\leq \beta \gamma_m h_{m+1,m} t (\varphi_{p+1})_t[\xi_1,\ldots,\xi_m].
\end{equation}
Hence,
\begin{equation}\label{eq.ddboundofdefintfrom1}
L_{p,m}(t) = \big( 1 + \Ono(|t\eta|^2) \big) \beta \gamma_m h_{m+1,m} t (\varphi_{p+1})_t[\xi_1,\ldots,\xi_m],
\end{equation}
using the notation $\Ono(|t\eta|^2)$ in the sense of $\Ono(|t\eta|)= \Ono(\max_j t|\eta_j|)$ for $t|\eta_j| \to 0$.
Following Proposition~\ref{thm.ddlamlowerbound} the choice of $\widetilde{\eta}_t$
is independent of $\xi_1,\ldots,\xi_m$,
and this carries over to the constant in~\eqref{eq.ddboundofdefintfrom1}.

Summarizing, we see that the defect integral can be computed exactly for the case of $H_m$ having real eigenvalues
(Corollary~\ref{cor.uppererrorboundreal}),
and a computable upper bound can be given which is tight for the case of $H_m$ having eigenvalues sufficiently close to the real axis
(Theorem~\ref{thm.upperbounderrorfull} and eq.~\eqref{eq.ddboundofdefintfrom1}).

The approach underlying Theorem~\ref{thm.upperbounderrorfull}
does not enable us to specify
the asymptotic constant in~\eqref{eq.ddboundofdefintfrom1}.
Therefore, we use the asymptotic expansion of the divided differences,
$|\exp_t[\lambda_1,\ldots,\lambda_m,0_p]|$ in~\eqref{eq.ddboundofdeffrom00},
derived in Appendix~\ref{sec.expansionofrho},
to discuss the asymptotic behavior of the defect norm $|\delta_{p,m}(t)|$ for $t\to 0$.
Theorem~\ref{thm.rho12asymptotic} from Appendix~\ref{sec.expansionofrho} implies
\begin{equation}\label{eq.defectrho}
\begin{aligned}
&| \exp_t[\lambda_1,\ldots,\lambda_m,0_p] | = \frac{t^{m+p-1}}{(m+p-1)!}
\exp\big(\rho_1 t + \rho_2 t^2/2 + \Ono(t^3) \big),\\
&\text{with}~~~ \rho_1 = \avg_p(\xi) ~~~\text{and}~~~ \rho_2 = \frac{\var_p(\xi) - \var_p(\eta)}{m+p+1}.
\end{aligned}
\end{equation}
Here, the asymptotics holds for $t\to 0$, $ \avg_p(\xi) = \sum_{j=1}^m \xi_j /(m+p)$
is the average, and
$ \var_p(\xi) = \big( \sum_{j=1}^m (\xi_j - \avg_p(\xi))^2 + p\avg_p(\xi)^2 \big)/ (m+p) $
is the variance of the sequence $\{\xi_1,\ldots,\xi_m,0_p\}$
and $\var_p(\eta)$ for the variance of the sequence $\{\eta_1,\ldots,\eta_m,0_p\}$.
\begin{remark}\label{rm.defderivationsH}
For $H_m$ with purely imaginary eigenvalues ($\lambda_j\in\ii\R$),
e.g. in the skew-Hermitian case,
the following asymptotic expansion for the defect is obtained
from~\eqref{eq.defectrho},
\footnote{It can be shown that the remainder is of
even order $\Ono(t^4)$ in this case.}
\begin{equation}\label{eq.defderivationvar}
|\delta_{p,m}(t)|
=\beta \gamma_m \frac{t^{m+p-1}}{(m+p-1)!}
\exp\Big( -\frac{\var_p(\eta)}{2(m+p+1)}t^2 + \Ono(t^3) \Big)
~~~\text{for}~~ t\to 0.
\end{equation}
\end{remark}
We use the expansion from~\eqref{eq.defectrho} for $| \exp_t[\lambda_1,\ldots,\lambda_m,0_p] |$
and $ \exp_t[\xi_1,\ldots,\xi_m,0_p] $ to obtain
\begin{equation}\label{eq.ddboundofdeffrom01}
|\delta_{p,m}(t)| = \exp\Big(-\frac{\var_p(\eta)}{2(m+p+1)}t^2 +\Ono(t^3)\Big) \cdot  \beta \gamma_m t^p (\varphi_p)_t[\xi_1,\ldots,\xi_m].
\end{equation}
Termwise integration of~\eqref{eq.ddboundofdeffrom01} and the proper prefactor
gives an asymptotic expansion for the defect integral $ L_{p,m}(t) $, similar to~\eqref{eq.ddboundofdefintfrom1},
\begin{equation}\label{eq.ddboundofdeffrom1}
L_{p,m}(t) = \Big(1-\frac{\var_p(\eta) (m+p)t^2}{2(m+p+1)(m+p+2)} +\Ono(t^3) \Big) \cdot  \beta h_{m+1,m}\gamma_m t(\varphi_{p+1})_t[\xi_1,\ldots,\xi_m].
\end{equation}
Omitting further details we state that~\eqref{eq.ddboundofdeffrom1}
is to be understood in an asymptotic sense with an remainder of $\Ono(t^3|\xi||\eta|^2+t^4|\eta|^4)$.
In contrast to~\eqref{eq.ddboundofdefintfrom1} the remainder is depending on $\xi$ terms
but~\eqref{eq.ddboundofdeffrom1} reveals further constants which can be relevant for practical applications.
\begin{remark}\label{rmk.isrealpartboundtight}
With~\eqref{eq.ddboundofdeffrom1} we obtain a computable
estimate for the relative deviation from the 
defect integral to the upper bound in~\eqref{eq.ddboundofdefintfrom0}.
The criterion
$$
\text{ac.est.}1(t):=\frac{\var_p(\eta) (m+p)t^2}{2(m+p+1)(m+p+2)} > 0.1,
$$
can indicate that a tighter estimate on the defect integral could
improve the error bound given in
Theorem~\ref{thm.upperbounderrorfull} in terms of accuracy.
A possible choice are quadrature estimates on the defect integral, see Subsection~\ref{sec.quadest} below.
\end{remark}
A similar criterion can be given for the accuracy of the upper bound,
\begin{equation}\label{eq.defectintupperboundEra}
L_{p,m}(t) \leq \beta h_{m+1,m}\gamma_m \frac{t^{m}}{(m+p)!},
\end{equation}
which appears in Corollary~\ref{thm.upperbounderrorfull2} (with $\xi_{\max}=0$) and~\cite[Theorem 1 and 2]{JAK19}.
With~\eqref{eq.defectrho}, and $\rho_1$ and $\rho_2$ given therein, the defect integral can be written as
\begin{equation}\label{eq.defectintrho}
L_{p,m}(t) = \beta h_{m+1,m}\gamma_m \frac{t^{m}}{(m+p)!}
\Big(1+\rho_1\frac{(m+p)t}{m+p+1} + (\rho_1^2+\rho_2)\frac{(m+p)t^2}{2(m+p+2)}
+\Ono(t^3)\Big)
\end{equation}
for $t\to 0$.
In contrast to the error bound in Corollary~\ref{thm.upperbounderrorfull2},
the formulas for $\rho_1$ and $\rho_2$ in~\eqref{eq.defectrho} require
the evaluation of the eigenvalues of $H_m$.
The following Proposition gives a formula for $\rho_1$ and $\rho_2$
which does not require computation of the eigenvalues of $H_m$
and can be evaluated on the fly.
\begin{proposition}[Evaluation of $\rho_1$ and $\rho_2$ in terms of entries of $H_m$]\label{thm.computeremainderera}
The coefficients $\rho_1$ and $\rho_2$ in~\eqref{eq.defectrho} can be rewritten as
\begin{align*}
&\rho_1 = \frac{\real(S_1)}{m+p},~~~\rho_2= \frac{\imag(S_1)^2-\real(S_1)^2}{(m+p)^2}+ \frac{\real(S_1^2+S_2) }{(m+p)(m+p+1)},~~~\text{with}\\
&S_1=\sum_{j=1}^m (H_m)_{j,j} ~~~\text{and}~~ S_2=\sum_{j=1}^m (H_m)^2_{j,j} + 2\sum_{j=1}^{m-1} (H_m)_{j+1,j}(H_m)_{j,j+1}.
\end{align*}
\end{proposition}
\begin{proof}
For the coefficients $\rho_1$ and $\rho_2$ we use~\eqref{eq.rho12}
with $ m \leftarrow m + p $
and $S_1$ and $S_2$ from~\eqref{eq.sl}.
For the nodes $\lambda_1,\ldots,\lambda_m,0_p$
(with $\lambda_1,\ldots,\lambda_m$ eigenvalues of $H_m$) we obtain
\begin{equation}\label{eq.S1S2derivEra}
\begin{aligned}
&S_1 = \sum_{j=1}^m \lambda_j = \text{Trace}(H_m) = \sum_{j=1}^m (H_m)_{j,j}~~~\text{and}\\
&S_2 =  \sum_{j=1}^m \lambda_j^2 = \text{Trace}(H_m^2) = \sum_{j=1}^m (H_m)^2_{j,j} + 2\sum_{j=1}^{m-1} (H_m)_{j+1,j}(H_m)_{j,j+1}.
\end{aligned}
\end{equation}
The identity for $\text{Trace}(H_m^2)$ in~\eqref{eq.S1S2derivEra}
holds true due to the upper Hessenberg structure of $H_m$.
\qed
\end{proof}
Following the proof of Theorem~\ref{thm.rho12asymptotic} we observe that the case $\rho_1=0$ is possible but results in $\rho_2\neq 0$.
\begin{remark}\label{rmk.isEratight}
With~\eqref{eq.defectintrho} and Proposition~\ref{thm.computeremainderera} we obtain a computable
estimate for the relative deviation from the 
defect integral to the upper bound in~\eqref{eq.defectintupperboundEra}.
The criterion
$$
\text{ac.est.}2(t):=\Big| \rho_1\frac{(m+p)t}{m+p+1} + (\rho_1^2+\rho_2)\frac{(m+p)t^2}{2(m+p+2)} \Big| > 0.1
$$
can indicate that a tighter estimate on the defect integral could
improve the error bound given in
Corollary~\ref{thm.upperbounderrorfull2} in terms of accuracy.
We refer to the error bound in Theorem~\ref{thm.upperbounderrorfull}
in case the eigenvalues of $H_m$ have a significant real part
(which can be observed via $\rho_1$).
\end{remark}

%
%
\subsection{Quadrature-based error estimates}\label{sec.quadest}
First we recapitulate some prior results.
In the dissipative case the integral formulation of the error
from Theorem~\ref{thm:errorint}
can be bounded via the defect integral via Corollary~\ref{thm:errorintexactarithmetic} up to round-off.
We conclude that the defect integral can be computed exactly
for the case of $H_m$ having real eigenvalues
(Corollary~\ref{cor.uppererrorboundreal}),
and a computable upper bound exists which is tight for the case of $H_m$ having eigenvalues sufficiently close to the real axis
(Theorem~\ref{thm.upperbounderrorfull} and eq.~\eqref{eq.ddboundofdefintfrom0}).

For the case of $H_m$ having eigenvalues with a significant imaginary part,
tight estimates are more difficult to obtain.
It can be favorable to approximate the defect integral~\eqref{eq.defint}
by quadrature to obtain an error estimate via Corollary~\ref{thm:errorintexactarithmetic}.
The aim of using quadrature is to obtain an error estimate which is tighter compared to previous upper norm bounds on the error.
In contrast to the proven upper error bounds given in Theorem~\ref{thm.upperbounderrorfull},
Corollary~\ref{cor.uppererrorboundreal} and~\ref{thm.upperbounderrorfull2}
the following quadrature estimates do not result in upper error bounds in general.
However, in many practical cases
such quadrature estimates turn out to be still reliable.

Here, some remarks on the defect are in order
to explain some subtleties with quadrature estimates for the defect integral
$ L_{p,m}(t) $.
We discuss a test problem with a skew-Hermitian matrix $A\in\C^{n\times n}$.
Following Remark~\ref{rmk:skewHermitiancaseLanczos} we choose $A=\ii B$ with a Hermitian matrix $B$,
in particulary, $B=\text{tridiag}(1,-2,1)\in\R^{n\times n}$ with $n=10\,000$.
The matrix $B$ is related to a finite difference discretization of the one-dimensional Laplacian operator
and $A$ corresponds to a free Schr{\"o}dinger type problem.
The eigenvalues $\sigma_j$, for $j=1,\ldots,n$, of $B$ are well studied, we obtain
\begin{equation}\label{eq.eigsdiscreteLaplace}
\sigma_j=4\sin(j\pi/(2(n+1)))^2
~~\text{with respective eigenvector $\psi_j\in\R^n$}.
\end{equation}
Here, $\mu_2(A)=0$, and the conditions of Corollary~\ref{thm:errorintexactarithmetic} hold.
For a given starting vector $v\in\C^n$
the time propagation for the discretized free Schr{\"o}dinger equation
is given by $\exp(tA)v$ and can be approximated by the Krylov propagator with $p=0$.
The following different cases for the starting vector $v$ will be discussed.
\begin{enumerate}[(a)]
\item\label{exdef.case1} Choose a random starting vector $v\in\R^n$.
\item\label{exdef.case2} Start close to a linear combination of eigenvectors,
$v = 10^6 \sum_{j=1}^{25} \psi_j  + \sum_{j=26}^n \psi_j$
for eigenvectors $\psi_j$ of the discretized Laplacian operator,~\eqref{eq.eigsdiscreteLaplace}.
\item\label{exdef.case3} Start close to a linear combination of eigenvectors which are more spread on the spectrum,
$v = 10^5 \sum_{j=1}^{20} \psi_j  + \sum_{j=21}^{n-20} \psi_j + 10^5 \sum_{j=n-19}^{n} \psi_j $
for eigenvectors $\psi_j$ of the discretized Laplacian operator,~\eqref{eq.eigsdiscreteLaplace}.
\end{enumerate}
In addition to the setting from~(\ref{exdef.case1})--(\ref{exdef.case3}) we normalize $v$, $\|v\|_2=1$.
The defect $\delta_{p,m}(t)$ for $p=0$ is computed in \textsc{Matlab},
using \texttt{expm} to evaluate the matrix exponential of
$H_m$ and divided differences for a fixed Krylov dimension $m=20$.

In Fig.~\ref{fig:skewHermitiandefect} we observe $|\delta_{p,m}(t)|=\Ono(t^{m-1})$ (for $t\to 0$)
up to $t\approx 10^{1}$ for the case~(\ref{exdef.case1})--(\ref{exdef.case3}).
The values of $|\delta_{p,m}(t)|$ in this time regime vary strongly between
among these cases.
We further remark that in the case~(\ref{exdef.case2}) for $t\geq 4\cdot 10^{1}$ the defect $|\delta_{p,m}(t)|$
behaves similar to the divided differences of the exponential over the first eigenvalues $\lambda_1^{(b)},\ldots,\lambda_4^{(b)}$
of $H_m$ with a proper prefactor.
This behavior occurs if eigenvalues of $H_m$ are clustered, in this case $\lambda_1^{(b)},\ldots,\lambda_4^{(b)}\approx 0$, and will be further discussed below, see Fig.~\ref{fig:ddclustermodel}.
For the case~(\ref{exdef.case3}) the eigenvalues of $H_m$ are clustered at $\approx 0$ and $ \approx 4 $.
Also in this case there is a time regime for which the defect behaves similar to a lower order function in $t$ with some additional oscillations.
(This may be explained by the existence of different eigenvalue clusters of the same size.)

As a conclusion from the example in Fig.3~\ref{fig:skewHermitiandefect},
we observe that quadrature of the defect 
can be relevant up to a time $t$
for which the quadrature based estimate of $\|l_{p,m}(t)\|_2$ (via the defect integral)
is equal to a given tolerance, see~\eqref{def.erroresttol}.
This regime of $t$ would depend on the choice of $\tol$ and additional factors such as $\beta$, $h_{m+1,m}$ etc.
which appear in the error bound from Corollary~\ref{thm:errorintexactarithmetic}.
Depending on parameters and the starting vector $v$ the defect can be highly oscillatory for relevant times $t$ and,
respectively, a quadrature estimate of the defect integral can be difficult to obtain.
Such effects seem to be relevant for special choices of starting vectors $v$, for example case~(\ref{exdef.case2}) and~(\ref{exdef.case3}).
The effect of $H_m$ having clustered eigenvalues and the prefactor used in Fig.~\ref{fig:skewHermitiandefect} ('$+$')
are explained in the following model problem, see Fig.~\ref{fig:ddclustermodel}.

\begin{figure}
\begin{tabular}{l}
\begin{overpic}
[width=0.75\textwidth]{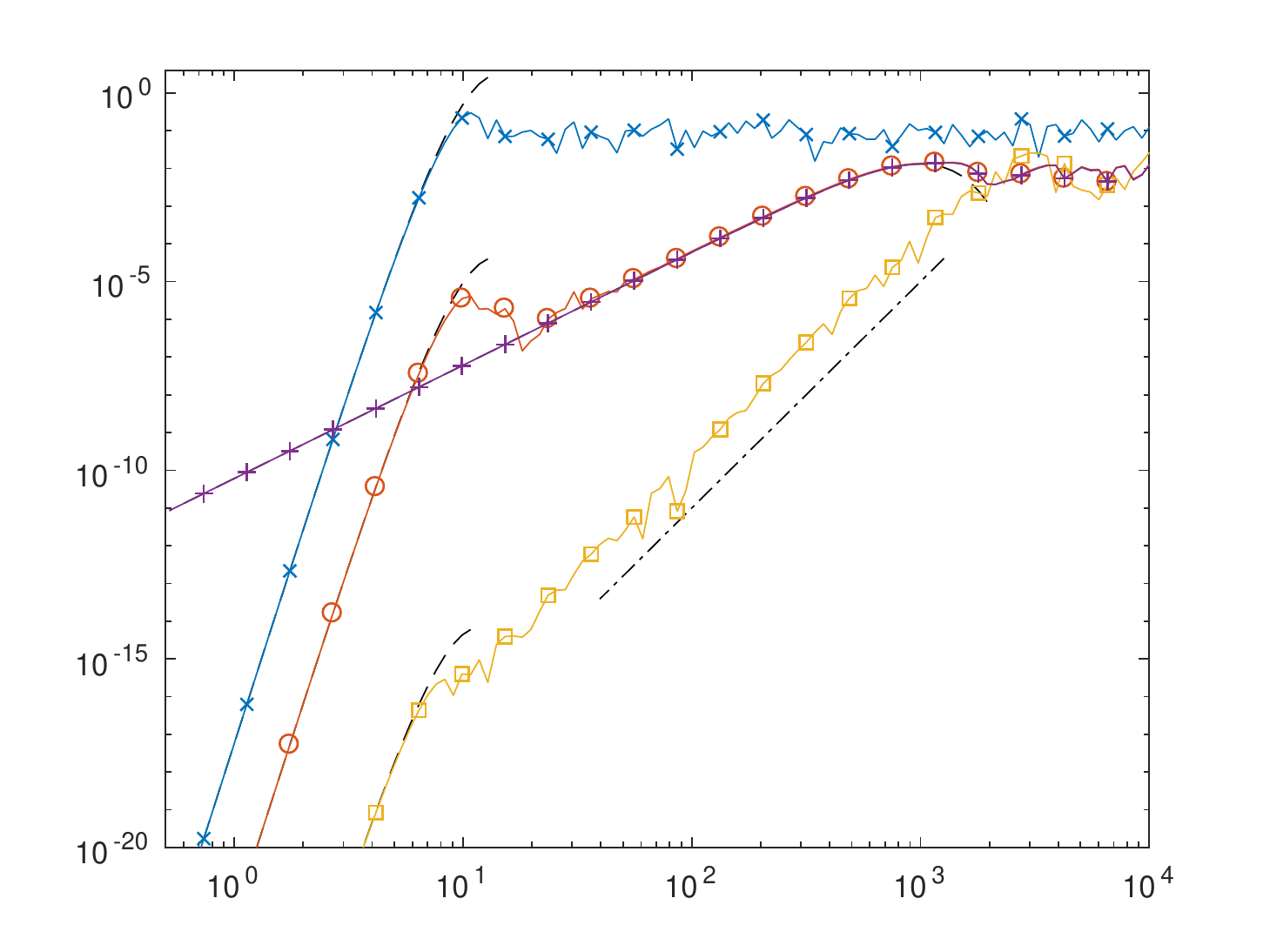}
\put(50,2){\small $t$}
\put(89,36){\begin{minipage}{5cm}
\small
$$
\begin{array}{|r|ccc|}
\hline
j & \lambda_j^{(a)} & \lambda_j^{(b)} & \lambda_j^{(c)}\\
\hline
1&    0.0002 &   0.0003  &  0.0001 \\
2&    0.0422 &   0.0026  &  0.0005 \\
3&    0.1360 &   0.0054  &  0.0013 \\
4&    0.2712 &   0.0108  &  0.0023 \\
5&    0.4743 &   0.3378  &  0.0032 \\
6&    0.6921 &   0.5763  &  0.0039 \\
7&    0.9440 &   0.8428  &  0.0054 \\
8&    1.2105 &   1.1343  &  0.9160 \\
9&    1.5049 &   1.4444  &  1.3768 \\
10&    1.8318 &   1.7660  &  1.7847 \\
11&    2.1456 &   2.0913  &  2.2385 \\
12&    2.4621 &   2.4124  &  2.6623 \\
13&    2.7540 &   2.7216  &  3.1348 \\
14&    3.0393 &   3.0112  &  3.9938 \\
15&    3.2997 &   3.2741  &  3.9961 \\
16&    3.5088 &   3.5038  &  3.9968 \\
17&    3.7091 &   3.6948  &  3.9977 \\
18&    3.8402 &   3.8423  &  3.9987 \\
19&    3.9510 &   3.9427  &  3.9995 \\
20&    3.9945 &   3.9935  &  3.9999 \\
\hline
\end{array}
$$
\end{minipage}}
\end{overpic}
\end{tabular}
\caption{The defect norm $|\delta_{p,m}(t)|$ ($p=0$, $m=20$) for the free Schr{\"o}dinger example with different choices of starting vector case~(\ref{exdef.case1}) ('$\times$'),
case~(\ref{exdef.case2}) ('$\circ$') and case~(\ref{exdef.case3}) ('$\Box$').
The table on the right-hand side shows eigenvalues $\lambda^{(\ast)}_1,\ldots,\lambda^{(\ast)}_m$ of $H_{m}$
for the different starting vectors, case~(\ref{exdef.case1})--(\ref{exdef.case3}).
For the case~(\ref{exdef.case2}) the divided differences over the clustered eigenvalues
$ \gamma_m  \big(\prod_{j=5}^{20} \lambda_j^{(b)}\big)^{-1} \exp_t[\lambda^{(b)}_1,\ldots,\lambda^{(b)}_4] $ is illustrated by ('$+$').
The asymptotic expansion of the divided differences for $t\to 0$ given in~\eqref{eq.defderivationvar} is illustrated using dashed lines.
The dash-dotted line is $\Ono(t^6)$.}
\label{fig:skewHermitiandefect}
\end{figure}

\paragraph{Divided differences with clustered nodes: an example.}
Choose $m=3$ with nodes $a_1=1.123,a_2=1.231,a_3=5.43$.
With this choice we obtain cluster of nodes, $a_1\approx a_2$.
For the given example we obtain $ | \exp_t[\ii a_2,\ii a_3] | \ll | \exp_t[\ii a_1,\ii a_2] | $
for $t$ large enough, hence,
using the recursive definition of the divided differences (see~\cite[eq. (B.24)]{Hi08} or others) we obtain
$$
| \exp_t[\ii a_1,\ii a_2,\ii a_3] | = \Big| \frac{ \exp_t[\ii a_2,\ii a_3] - \exp_t[\ii a_1,\ii a_2] }{a_3-a_1}\Big|
\approx  \Big| \frac{ \exp_t[\ii a_1,\ii a_2] }{a_3-a_1}\Big|,~~~\text{for larger $t$.}
$$
This example is illustrated in Fig.~\ref{fig:ddclustermodel}.
This behavior can be generalized for a larger number of nodes and is also
observed in Fig.~\ref{fig:skewHermitiandefect}.

\begin{figure}
\centering
\begin{overpic}
[width=0.8\textwidth]{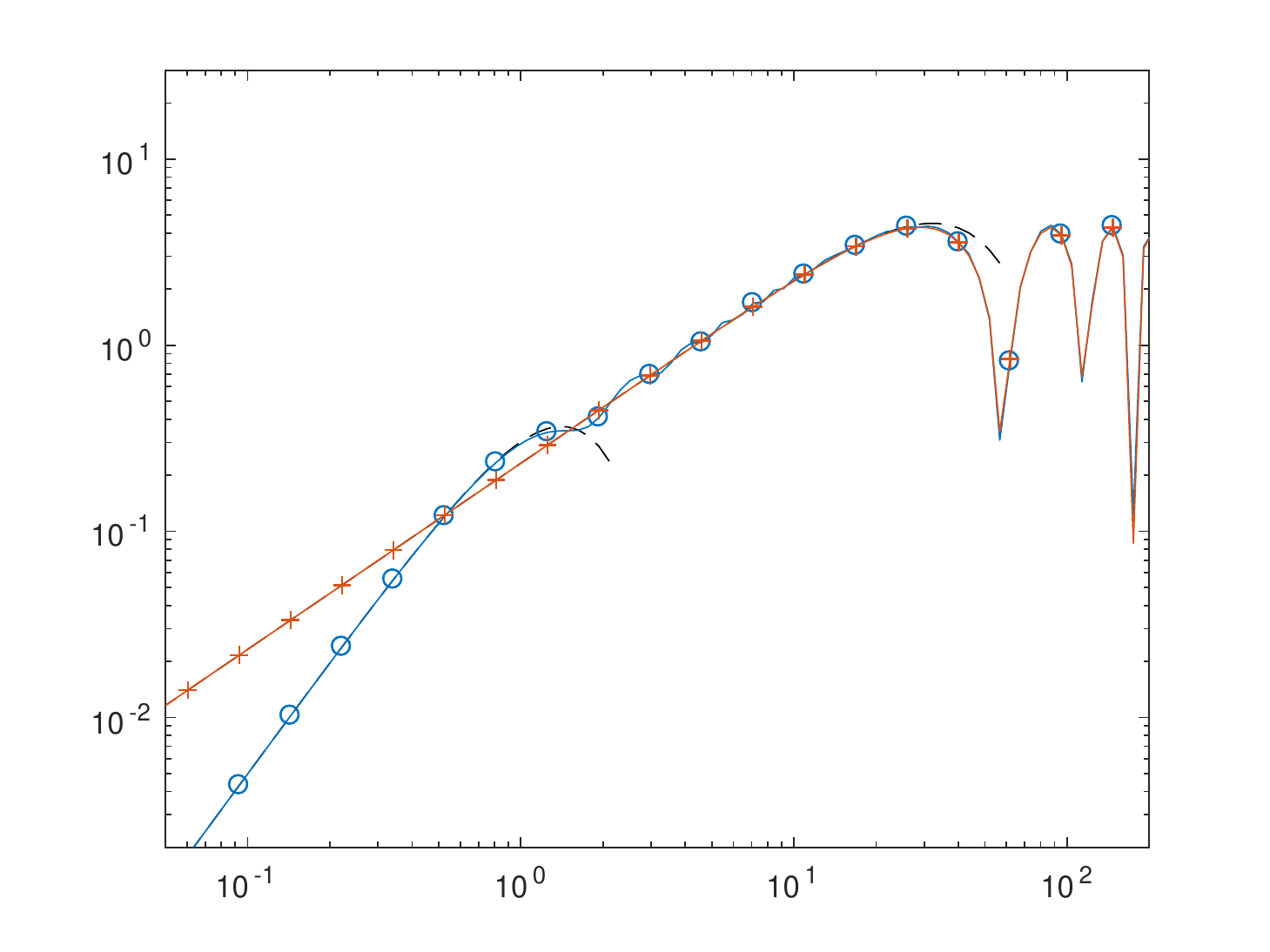}
\put(50,2){\small $t$}
\end{overpic}
\caption{The divided differences $| \exp_t[\ii a_1,\ii a_2,\ii a_3] |$ ('$\circ$')
and $| \exp_t[\ii a_1,\ii a_2] |/|a_3-a_1| $  ('$+$')
for the choice of $a_1,a_2,a_3$ given in the text.
The asymptotic expansion of the divided differences for $t\to 0$ given in~\eqref{eq.defderivationvar} is illustrated using dashed lines.}
\label{fig:ddclustermodel}
\end{figure}

\paragraph{Quadrature estimates for the defect integral.}
With the previous observations on the defect
we now discuss different quadrature-based estimates.

We can extend the result of the generalized residual estimate,
which was introduced in~\cite{HLS98}
and appeared in a similar manner in~\cite{DGK98,Sa92,Lu08,BGH13},
to $\varphi$-functions using the defect integral
according to Corollary~\ref{thm:errorintexactarithmetic}.
\begin{remark}[Generalized residual estimate, see also~\cite{HLS98}]\label{thm:generalizedresidualest}
Applying the right-endpoint rectangle rule we have
$$
\int_0^t |\delta_{p,m}(s)| \,\dd s \approx t|\delta_{p,m}(t)|,
$$
and with Corollary~\ref{thm:errorintexactarithmetic} (and $\delta_{p,m}(t)$ given in~\eqref{eq.defpm}) we obtain the error estimate
$$
\|l_{p,m}(t)\|_2 \approx h_{m+1,m} t^{1-p} |\delta_{p,m}(t)| = \beta h_{m+1,m} t |e_m^\ast \varphi_p(t H_m) e_1|  .
$$
\end{remark}
Assume that $\max_{s\in[0,t]} |\delta_{p,m}(t)|=|\delta_{p,m}(t)|$,
e.g. $|\delta_{p,m}(t)|$ is monotonically increasing in $t$. Then,
$$
\int_0^t |\delta_{p,m}(s)| \,\dd s
\leq t \max_{s\in[0,t]} |\delta_{p,m}(t)|
= t |\delta_{p,m}(t)|.
$$
In this case the generalized residual estimate from
Remark~\ref{thm:generalizedresidualest} results in an upper bound on the error norm.

In the most general case the defect is of a high order for $t\to0$ and in a relevant time regime,
see also Fig.~\ref{fig:skewHermitiandefect} case~(\ref{exdef.case1}) and previous remarks.
Then the defect is a higher order function and the right-endpoint quadrature does result in an upper bound but is not tight.
In this case we can improve the estimate by a prefactor
depending on the {\em effective order} defined in Appendix~\ref{sec.expansionofrho}.
If the defect is sufficiently smooth in a relevant time regime this results in a tight upper bound on the error norm.

\begin{remark}[Effective order estimate, see also~\cite{JAK19}]\label{thm:effectiveorderest}
Denote $f(t) = | \exp_t[\lambda_1,\ldots,\lambda_m,0_p] |$ for the time-dependent part of the defect
with eigenvalues $\lambda_1,\ldots,\lambda_m$ of $H_m$.
Assume $f(t)>0$ for a sufficiently small time regime $t>0$.
We consider the effective order $\rho(t)$ to be defined for the divided differences $f(t)$ as given in~\eqref{def:effectiveorder}.
With the following estimate for the integral of the defect,
$$
\int_0^t |\delta_{p,m}(s)| \,\dd s \approx \frac{t}{\rho(t)+1}|\delta_{p,m}(t)|,
$$
and from Corollary~\ref{thm:errorintexactarithmetic} (with $\delta_{p,m}(t)$ given in~\eqref{eq.defpm}) we obtain
$$
\|l_{p,m}(t)\|_2 \approx h_{m+1,m} \frac{ t^{1-p} }{ \rho(t)+1 } |\delta_{p,m}(t)|
= \beta h_{m+1,m} \frac{ t }{ \rho(t)+1 } |e_m^\ast \varphi_p(t H_m) e_1| .
$$
\end{remark}
In~\cite{JAK19} the effective order is defined for $|e_m^\ast \ee^{tH_m} e_1|$ ($p=0$)
which is equivalent to the definition via the divided differences of $f(t)$.
(This follows from Corollary~\ref{thm.defectequal}
and the definition of the effective order which is independent
of a constant prefactor.)

Some of the following observations already appeared in~\cite{JAK19}.
The quadrature scheme in Remark~\ref{thm:effectiveorderest}
is motivated by the following relation of the effective order and the integral of the divided differences $f(t)$.
From eq.~\eqref{def:effectiveorder},
\begin{equation*}
f(t)=\frac{f'(t)\,t}{\rho(t)}.
\end{equation*}
Integration and application of the mean value theorem shows the existence of $t^\ast\in[0,t]$ with
\begin{equation*}
\int_0^t f(s)\,\dd s= \frac{1}{\rho(t^\ast)} \int_0^t f'(s)\,s\,\dd s,
\end{equation*}
and integration by parts gives
\begin{equation}\label{eq.intdefefforder}
\int_0^t f(s) \,\dd s = \frac{t f(t)}{1+\rho(t^\ast)}.
\end{equation}
This result can passed over to the integral of the defect.

Assume the effective order is monotonically decreasing in for $t$ small enough,
$\min_{s\in(0,t]}\rho(s) = \rho(t) \geq 0$.
This holds in an asymptotic regime for the dissipative case up to round-off,
see also Theorem~\ref{thm.rho12asymptotic} with the real parts $\xi_1,\ldots,\xi_m$
of the eigenvalues of $H_m$ being non-positive.
With~\eqref{eq.intdefefforder} and the assumption
$0\leq \rho(t) \leq \rho(s) \leq m+p-1 =\rho(0+)$ for $s\in[0,t]$,
we inclose the integral of the defect by
\begin{equation}\label{defintbounds}
\tfrac{t}{m}\,|\delta_{p,m}(t)| \leq \int_0^t |\delta_{p,m}(s)| \,\dd s
\leq \tfrac{t}{\rho(t)+1}\,|\delta_{p,m}(t)| \leq t\,|\delta_{p,m}(t)|.
\end{equation}
Combining~\eqref{defintbounds} and Corollary~\ref{thm:errorintexactarithmetic}
we obtain the upper bound
$$
 \|l_{p,m}(t)\|_2\leq \tfrac{ h_{m+1,m} t^{1-p} }{\rho(t)+1}\cdot|\delta_{p,m}(t)| \leq  h_{m+1,m} t^{1-p} \cdot|\delta_{p,m}(t)|.
$$
A computable expression for the effective order was given in~\cite[eq. (6.10)]{JAK19}.
This result can be generalized to the case $p\in\N_0$,
\begin{align*}
\rho(t) =
\left\{
\begin{array}{ll}
t \real\big( (H_m)_{m,m} + (H_m)_{m,m-1} (y_{p,m}(t))_{m-1}/(y_{p,m}(t))_m \big)~~&\text{for}~\,p=0,~~\text{and}\\
\real( (y_{p-1,m}(t))_{m}/(y_{p,m}(t))_m )~~&\text{for}~\,p\in\N,
\end{array}\right.
\end{align*}
with $y_{p,m}(t)\in\C^m$ from~\eqref{eq.varphi}.
The expression for the case $p\in\N$ can be obtained by~\cite[eq. (6.10)]{JAK19}
applied on the representation $|e_{m+p}^\ast \ee^{t\widetilde{H}_m} e_1|$ for the defect
(\ref{thm.defectequal.extendH}. in Corollary~\ref{thm.defectequal})
and making use of the special structure of $\widetilde{H}_m$,
$\beta e_{m+p}^\ast \ee^{t\widetilde{H}_m} e_1 = t^p (y_{p,m}(t))_m$ (see Corollary~\ref{thm.defectequal})
and $\beta e_{m+p-1}^\ast \ee^{t\widetilde{H}_m} e_1 = t^{p-1} (y_{p-1,m}(t))_m$ (see~\cite[Corollary 1]{Si98}).

As illustrated in Fig.~\ref{fig:skewHermitiandefect} the defect can be highly oscillatory in a relevant time regime,
especially for specific starting vectors,
and in this case the quadrature estimates should be handled with care.

\subsection{A stopping criterion for lucky breakdown.}\label{sec.luckybreakdown}

The special case $h_{k+1,k}=0$ during the construction of the Krylov subspace
is considered to be a {\em lucky breakdown,}
a breakdown of the Arnoldi or Lanczos iteration with the benefit of an exact approximation of $\varphi_p(tA)v$ for any $t>0$
via the Krylov subspace $\Kry_k(A,v)$.
In floating point arithmetic the lucky breakdown results in $h_{k+1,k}\approx 0$
and can lead to stability issues if the Arnoldi or Lanczos method is not stopped properly.
The condition that the Krylov propagator is exact is not exactly determinable
in floating point arithmetic but
can be weakened to the error condition in~\eqref{def.erroresttol} for a given tolerance~$\tol$ per unit step.
With this approach we introduce a stopping criterion
which can be applied on the fly to detect a lucky breakdown
and satisfies an error bound.
This does not depend on any a~priori information as long the tolerance~$\tol$
is chosen properly so that round-off errors can be neglected,
see remarks before Corollary~\ref{thm:errorintexactarithmetic}.

\begin{proposition}\label{thm.luckybreakdownstop}
Let $\mu_2(A)\leq 0$ and assume that round-off errors are sufficiently small, see Corollary~\ref{thm:errorintexactarithmetic}.
Let \,$\tol$ be a given tolerance and 
\begin{equation}\label{eq.stoppingcriteria}
\frac{\beta h_{k+1,k}}{(p+1)!}\leq \tol
\end{equation}
be satisfied at the $k$-th step of the Arnoldi or Lanczos iteration.
Then the iteration can be stopped and the Krylov subspace $\Kry_k(A,v)$
can be used to approximate the vector $\varphi_p(tA)v$
with a respective error per unit step $\|l_{p,k}(t)\|_2 \leq t\cdot \tol$.
\end{proposition}
\begin{proof}
We use the upper bound on the error norm from Corollary~\ref{thm:errorintexactarithmetic},
\begin{equation}\label{eq.errestluckyerrorint}
 \|l_{p,k}(t)\|_2 \leq \frac{ h_{k+1,k}}{t^p}\int_0^t |\delta_{p,k}(s)| \,\dd s.
\end{equation}
To obtain a uniform bound on the defect integral we use
\begin{equation}\label{eq.deftophinorm}
|\delta_{p,k}(t)| \leq \beta t^p \|e_k\|_2 \|\varphi_p(t H_k) e_1\|_2
= \beta t^p \|\varphi_p(t H_k) e_1\|_2.
\end{equation}
\begin{itemize}
\item
For $p>0$ we apply the integral representation~\eqref{defphiint} on $\varphi_p(t H_m) e_1$
to obtain the upper bound
\begin{equation}\label{eq.interrestnormbound3x}
\|\varphi_p(t H_m) e_1\|_2
\leq \frac{\max_{s\in[0,t]} \|\ee^{s H_m}\|_2}{(p-1)!} \int_0^1\theta^{p-1}\,\dd\theta
= \frac{\max_{s\in[0,t]} \|\ee^{s H_m}\|_2 }{p!}.
\end{equation}
\item
For $p=0$ the analogous result is directly obtained:
Combine~\eqref{eq.deftophinorm} and~\eqref{eq.interrestnormbound3x}
with $\|\ee^{sH_k}\|_2 \leq \ee^{t\mu_2(H_k)} \leq \ee^{t\mu_2(A)}$ up to round-off
and $\mu_2(A)\leq 0$, giving
$$
|\delta_{p,k}(t)| \leq \beta \frac{t^p }{p!},~~~\text{and}~~\int_0^t |\delta_{p,k}(s)| \,\dd s \leq \beta \frac{t^{p+1} }{(p+1)!}.
$$
\end{itemize}
Together with~\eqref{eq.errestluckyerrorint} and~\eqref{eq.stoppingcriteria} we conclude $\|l_{p,k}(t)\|_2 \leq t\cdot \tol$.
\qed
\end{proof}


\section{Numerical experiments}\label{sec.numexp}
The notation for the error $l_{p,m}(t)$, the estimate of the error norm $\zeta_{p,m}(t)$ and the tolerance $\tol$
have been introduced in~\eqref{def.errorpm} and~\eqref{def.erroresttol}.
The notation $\zeta_{p,m}$ will be used for different choices of error estimates discussed in the previous section.
Theorem~\ref{thm.upperbounderrorfull} and Corollary~\ref{thm.upperbounderrorfull2}
result in upper bounds on the error norm, $ \|l_{p,m}(t)\|_2 \leq \zeta_{p,m}(t) $ .
The quadrature-based error estimates given in Remark~\ref{thm:generalizedresidualest} and~\ref{thm:effectiveorderest}
result in estimates for the error norm, $ \|l_{p,m}(t)\|_2 \approx \zeta_{p,m}(t) $,
and with additional conditions also give upper bounds.
For a fixed tolerance $\tol$ we use the notation $t(m)$
for the smallest time $t$ with $\zeta_{p,m}(t) = t\cdot \tol$, see~\eqref{def.erroresttol}.
This choice of $t(m)$ helps us to verify the tested error estimates for a time $t$ which is of the most practical interest.
With the help of a reference solution the true error norm per unit step can be tested by $ \|l_{p,m}(t(m))\|_2 / t(m) $.
For the numerical experiments we focus on the most prominent case $p=0$
and simplify the notation by writing $\delta_m(t) = \delta_{0,m}(t)$, $l_{m}(t) = l_{0,m}(t)$ and $\zeta_{m}(t) = \zeta_{0,m}(t)$.

\subsection{Convection-diffusion equation}\label{sec.numexp.cv}
Consider the following two-dimensional convection-diffusion equation with
$t\geq 0$ and $x\in[0,1]^2$,
\begin{equation}\label{eq.cdcont}
\partial_t u = L u,~~~\text{with}~~L= \Delta + \nu (\partial_{x_1} + \partial_{x_2} ) ,~~~u=u(t,x),~\nu\in\R.
\end{equation}
Let $A\in\R^{n\times n}$ be obtained by the two-dimensional finite difference discretization of the operator $L$ in~\eqref{eq.cdcont}
with zero dirichlet boundary conditions and $N=500$ inner mesh points in each spatial direction,
hence, $n=N^2$.
This test problem is similar to other convection-diffusion equations
appearing in the study of Krylov subspace methods, see also~\cite{JAK19,EE06,FGS14a,BK19} and others.

The symmetric case with $\nu=0$ results in the Heat equation
and has already been discussed in~\cite{JAK19}.
For the convection parameter we choose $\nu=100,500$ which results in a non-normal matrix $A$.
Considering the spectrum of $A$ the case $\nu=100$ is closer to the Hermitian case
and $\nu=500$ is closer to the skew-Hermitian case.
We remark that for the real matrix $A\in\R^{n\times n}$ the terms Hermitian and symmetric can be used in an equivalent manner.
In both cases the numerical range of $A$ is in the left complex plane, $ \mu_2(A)\leq 0 $.

We discuss error estimates for the case $p=0$, hence, $\ee^{tA}v$
is approximated in the Krylov subspace $\Kry_m(A,v)$, see~\eqref{eq.Krylovprpagatorexp}.
As a starting vector we choose the normalized vector $v=(1,\ldots,1)^\ast\in\R^n$.
The error estimates given in Theorem~\ref{thm.upperbounderrorfull}, Corollary~\ref{thm.upperbounderrorfull2}
and Remark~\ref{thm:generalizedresidualest} and~\ref{thm:effectiveorderest} are compared
for this problem in Fig.~\ref{fig:numexcv100u500}.

For the case $\nu=100$ the eigenvalues of $H_m$ have a negligible imaginary part
and the upper bound given in Theorem~\ref{thm.upperbounderrorfull} shows tight results.
For $\nu=500$ and larger choices of $m$ this bound is less tight.
The criterion $\text{ac.est.}1(t)$ given in Remark~\ref{rmk.isrealpartboundtight}
is evaluated for $\nu=100,500$ with $t(m)$ corresponding to Theorem~\ref{thm.upperbounderrorfull} (see caption of Fig.~\ref{fig:numexcv100u500}).
For $\nu=100$ we obtain $\text{ac.est.}1(t(m))<0.1$ for any $m$ tested
and for $\nu=500$ the smallest $m$ with $\text{ac.est.}1(t(m))>0.1$ is $m=45$.
The upper bound of Corollary~\ref{thm.upperbounderrorfull2} is applied with $\xi_{\max}=0$
(the effect of $\xi_{\max}$ is negligible in this case).
Similar to the criterion $\text{ac.est.}1(t)$, we test $\text{ac.est.}2(t)$ given in Remark~\ref{rmk.isEratight}
for $t(m)$ corresponding Corollary~\ref{thm.upperbounderrorfull2}.
The smallest $m$ with $\text{ac.est.}2(t(m))>0.1$ is $m=9$ and $m=10$ for $\nu=100$ and $\nu=500$, respectively.
The quadrature-based error estimates given in Remark~\ref{thm:generalizedresidualest} and~\ref{thm:effectiveorderest}
and both result in upper bounds on the error norm for this example,
whereas the effective order estimate in Remark~\ref{thm:effectiveorderest} results in a tighter bound.

\begin{figure}
\centering
\begin{tabular}{cc}
\hspace{0.85cm}
\begin{overpic}
[width=0.48\textwidth]{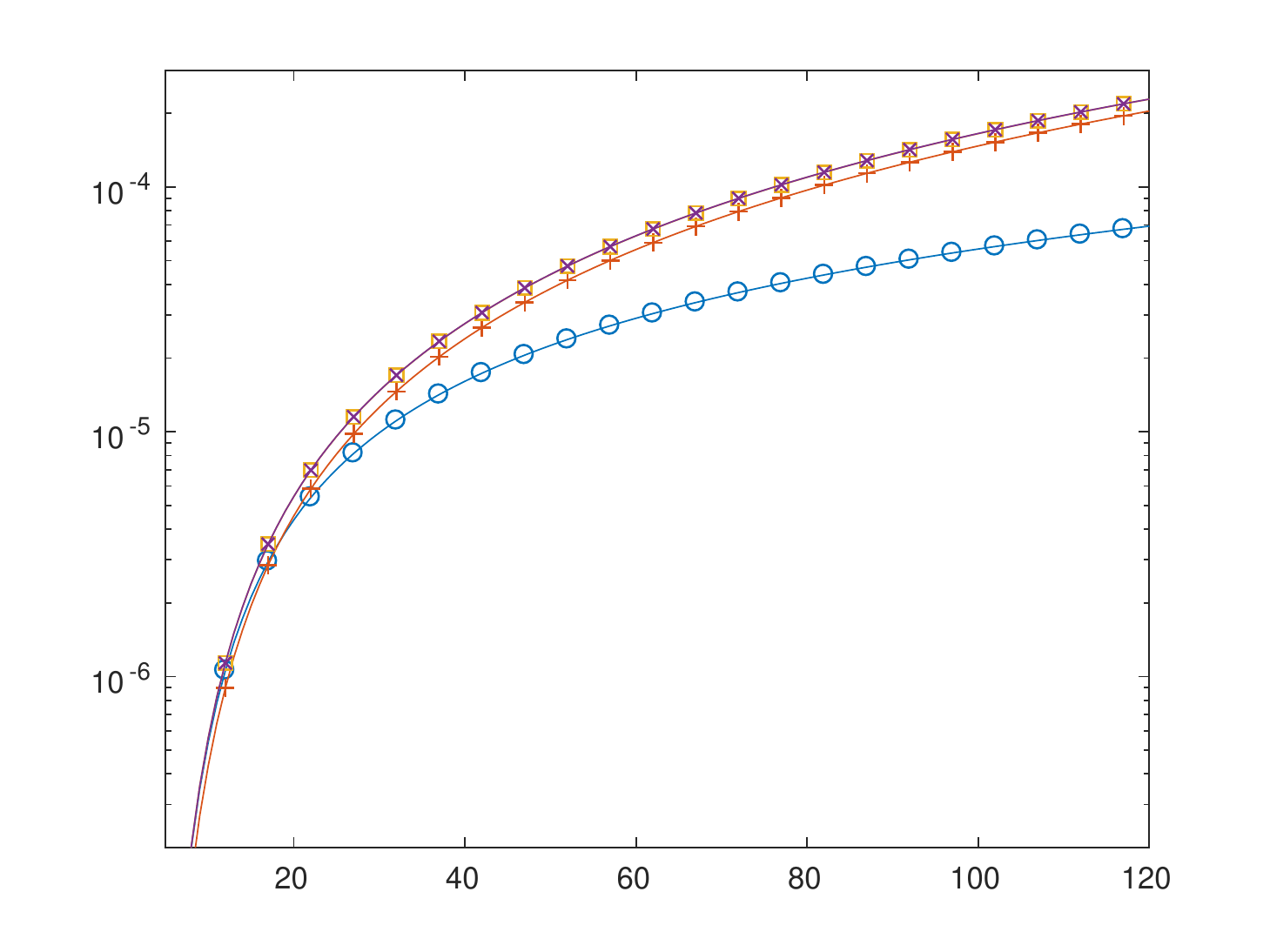}
\put(41,73){ $\nu=100$}
\put(-5,35){\small $t(m)$}
\end{overpic}
&
\hspace{-0.8cm}
\begin{overpic}
[width=0.48\textwidth]{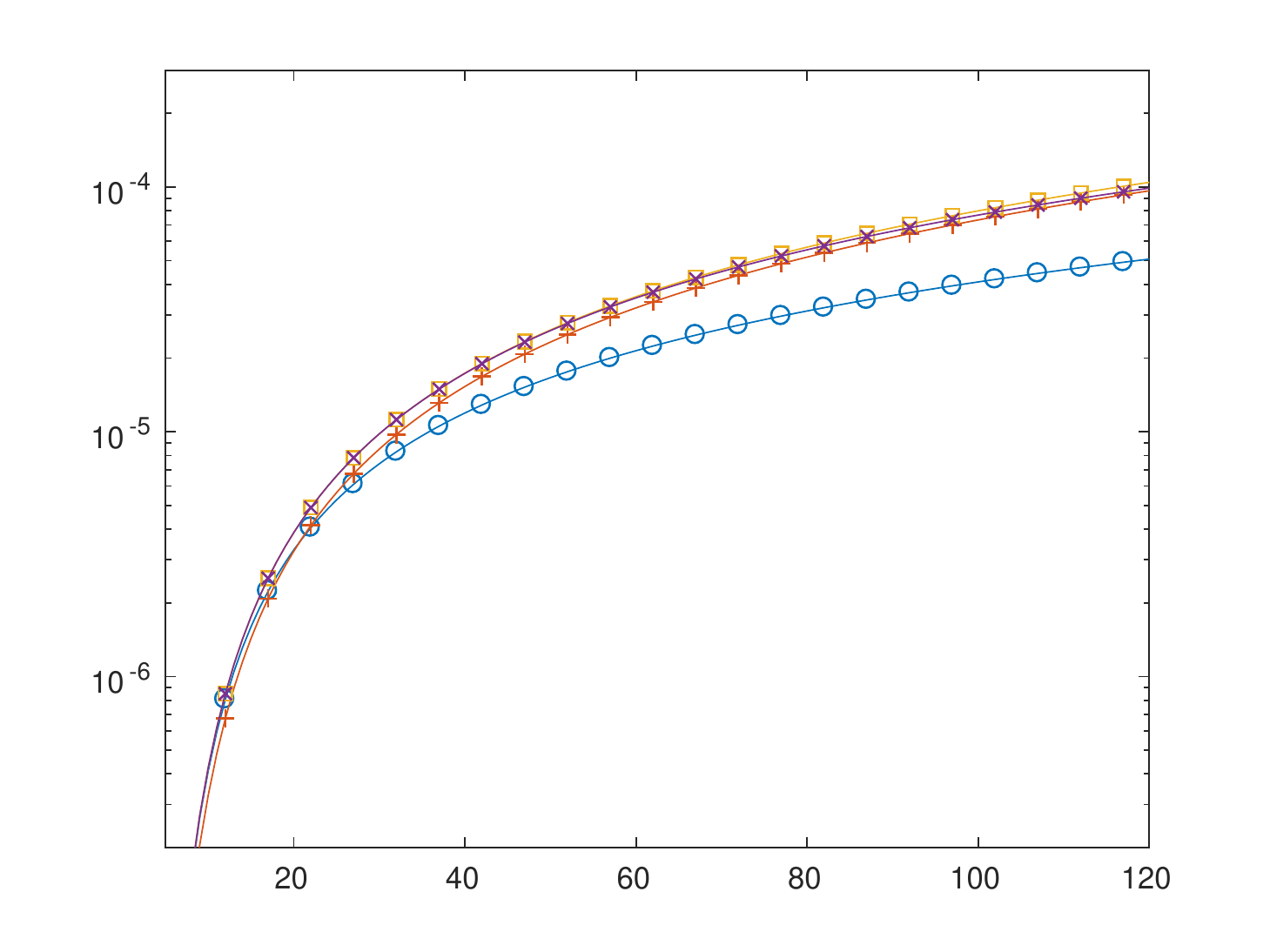}
\put(41,73){ $\nu=500$}
\end{overpic}
\\
\hspace{0.85cm}
\begin{overpic}
[width=0.48\textwidth]{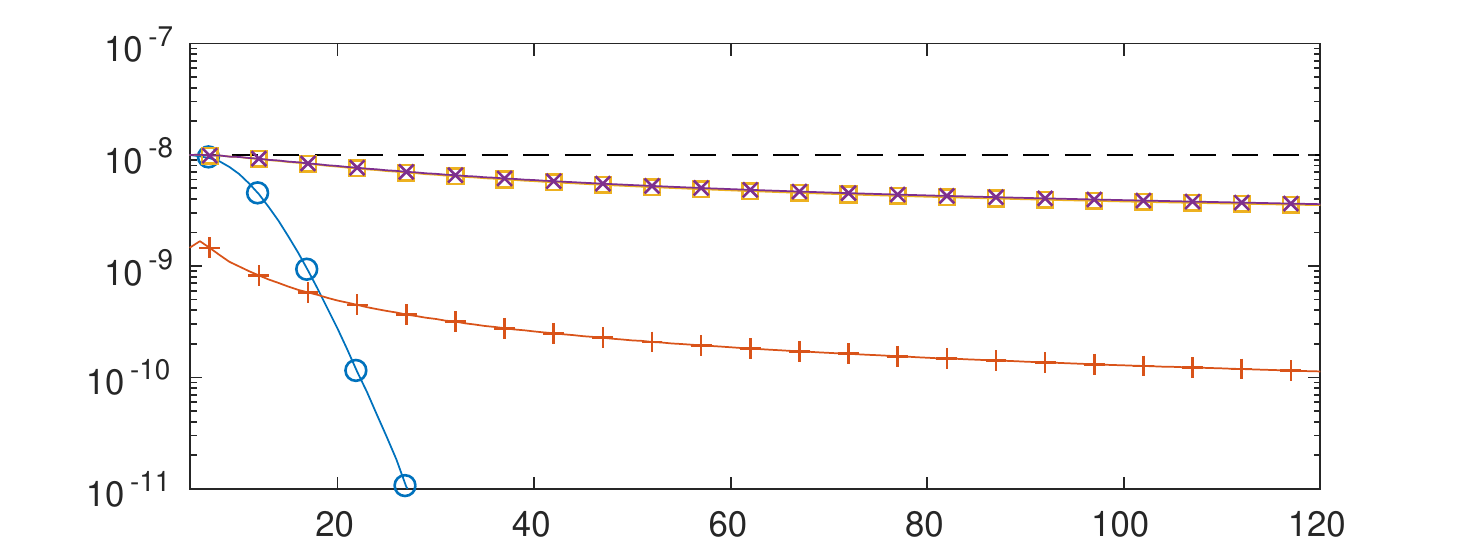}
\put(-15,15){\small $\frac{\|l_m(t(m))\|_2}{t(m)}$}
\put(45,28){\small $\tol=10^{-8}$}
\put(50,-3){\small $m$}
\end{overpic}
&
\hspace{-0.8cm}
\begin{overpic}
[width=0.48\textwidth]{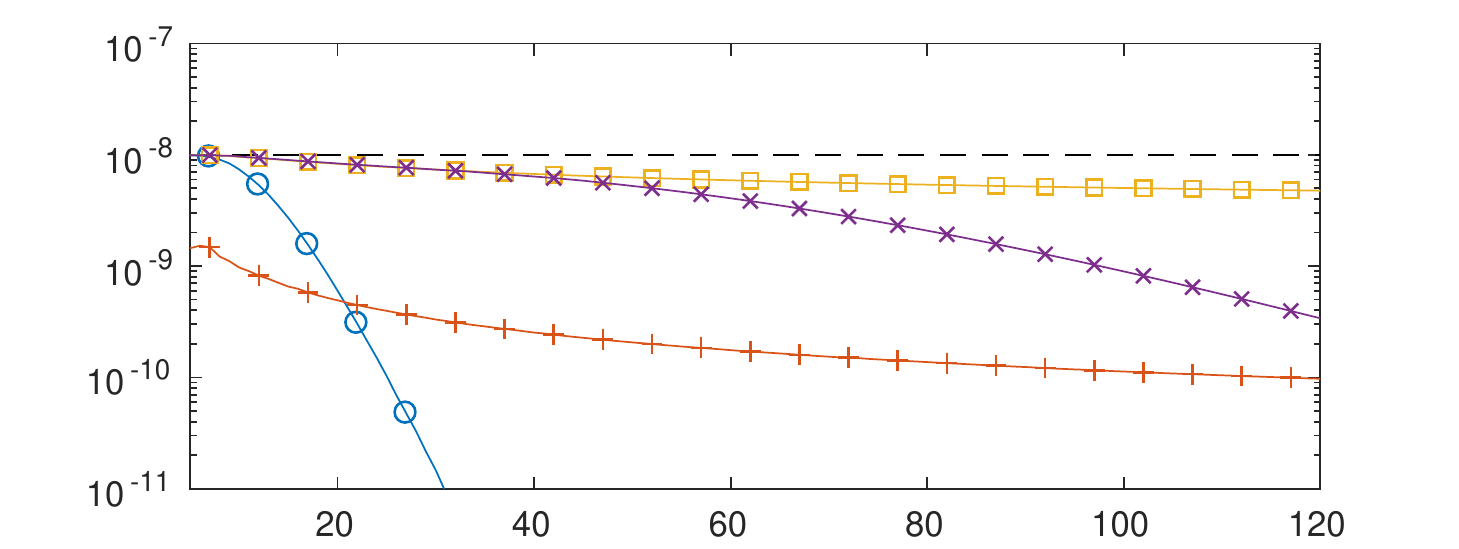}
\put(45,28){\small $\tol=10^{-8}$}
\put(50,-3){\small $m$}
\end{overpic}
\end{tabular}
\caption{The convection-diffusion problem~\eqref{eq.cdcont} for the parameter $\nu=100$ (left) and $\nu=500$ (right).
The top row shows the time $t(m)$ which is the smallest $t$ so that $ \zeta_{m}(t) = t\cdot \tol $
for $\tol=10^{-8}$ and $\zeta_{m}$
being the upper norm bound given in Theorem~\ref{thm.upperbounderrorfull} ('$\times$'),
Corollary~\ref{thm.upperbounderrorfull2} ('$\circ$'), the generalized residual estimate given in Remark~\ref{thm:generalizedresidualest} ('$+$')
and the effective order estimate given in Remark~\ref{thm:effectiveorderest} ('$\Box$').
The bottom row shows the true error per unit step, $\|l_{m}(t(m))\|_2/t(m)$, for the
time $t(m)$ as chosen above.}
\label{fig:numexcv100u500}
\end{figure}
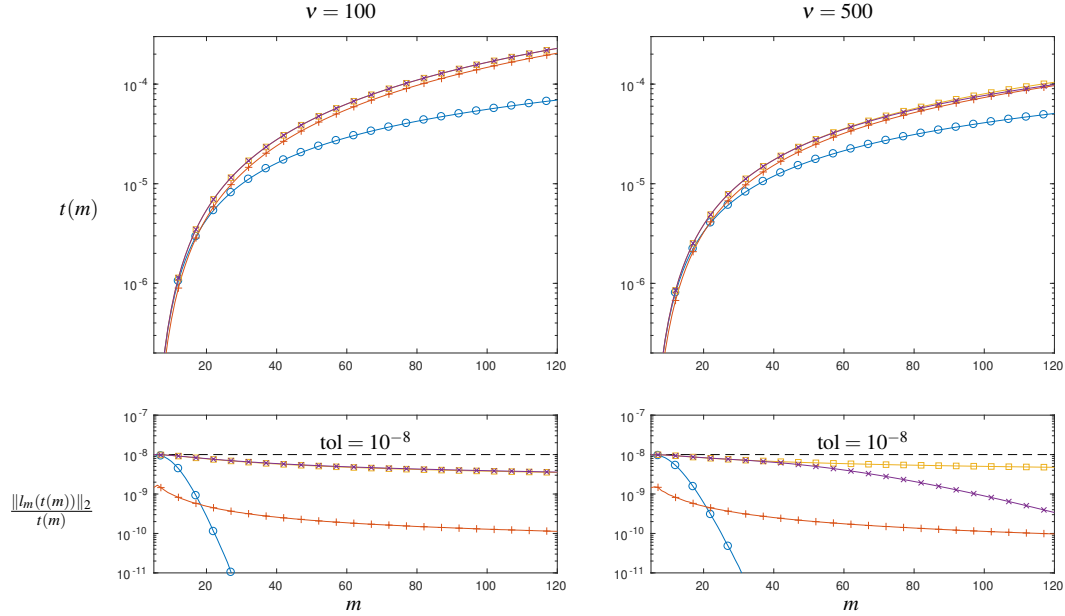

\subsection{Free Schr{\"o}dinger equation with a double well potential}\label{sec.numexp.fS}
Consider the one-dimensional free Schr{\"o}dinger equation with a double well potential.
\begin{equation}\label{eq.schrocont}
\partial_t \psi= -\ii H\psi,~~~\text{with}~~H = \Delta + V,~~~\psi=\psi(t,x)\in\C,~V=V(x)\in\R,
\end{equation}
for $t\geq 0$, $x\in[-10,10]$ and $V(x) = x^4-15 x^2$.
Let $B\in\C^{n\times n}$ be the discretized version of the Hamiltonian operator $H$ in~\eqref{eq.schrocont}
with periodic boundary conditions using a finite difference scheme with a mesh of size $n=10000$.
With $B$ Hermitian, the full problem $A=-\ii B$ is skew-Hermitian
(see Remark~\ref{rmk:skewHermitiancaseLanczos}) and we obtain $\mu_2(A)=0$.
For the initial state of~\eqref{eq.schrocont} we choose a Gaussian wavepacket,
$$
\psi(t=0,x)=(0.2\pi)^{-1/4}\exp(-(x+2.5)^2/(0.4)),
$$
which is evaluated on the mesh and normalized to obtain a discrete starting vector $v\in\R^n$.
This problem also appears in~\cite{IKS19,Si19}.

Similar to the previous subsection we discuss error estimates for the case $p=0$, hence, the Krylov approximation of $ \ee^{-\ii t B} v $.
The implementation of the skew-Hermitian problem is described in Remark~\ref{rmk:skewHermitiancaseLanczos}.
In Fig.~\ref{fig:numexfS} the upper bound given in Theorem~\ref{thm.upperbounderrorfull} and Corollary~\ref{thm.upperbounderrorfull2},
which coincidence in the skew-Hermitian case,
and the error estimates given in Remark~\ref{thm:generalizedresidualest} and~\ref{thm:effectiveorderest} are compared.
With our choice starting vector the matrix $H_m$ does have clustered eigenvalues, see also Subsection~\ref{sec.quadest}.
The defect $\delta_{m}(t)$, which is presented in the lower right corner of Fig.~\ref{fig:numexfS},
does have an oscillatory behavior which is luckily not in the relevant time regime.
Therefore, with our choice of tolerance $\tol=10^{-8}$ the quadrature-based error estimates are still valid.
In other cases this oscillatory behavior of the defect can lead to failure of the error estimates
given in Remark~\ref{thm:generalizedresidualest} and~\ref{thm:effectiveorderest}.
The upper bound given in Corollary~\ref{thm.upperbounderrorfull2} is reliable but not tight for this example
which can also be explained by the loss of order of the defect caused by the starting vector.

\begin{figure}
\hspace{0.36cm}
\begin{tabular}{l}
\begin{overpic}
[width=0.65\textwidth]{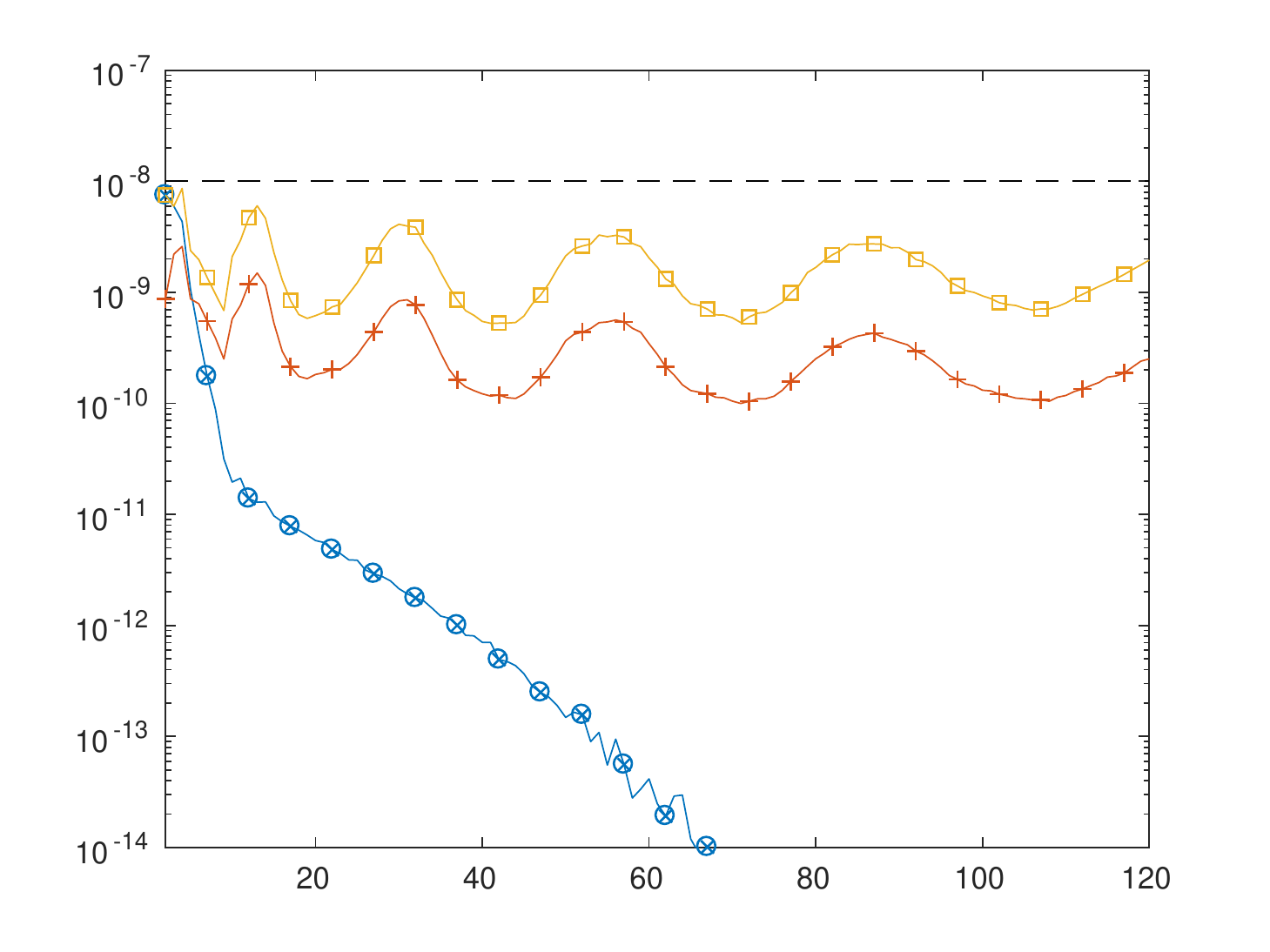}
\put(45,62){\small $\tol=10^{-8}$}
\put(50,0){\small $m$}
\put(-7,37){\small $\frac{\|l_m(t(m))\|_2}{t(m)}$}
\put(103,5){\includegraphics[width=0.29\textwidth]{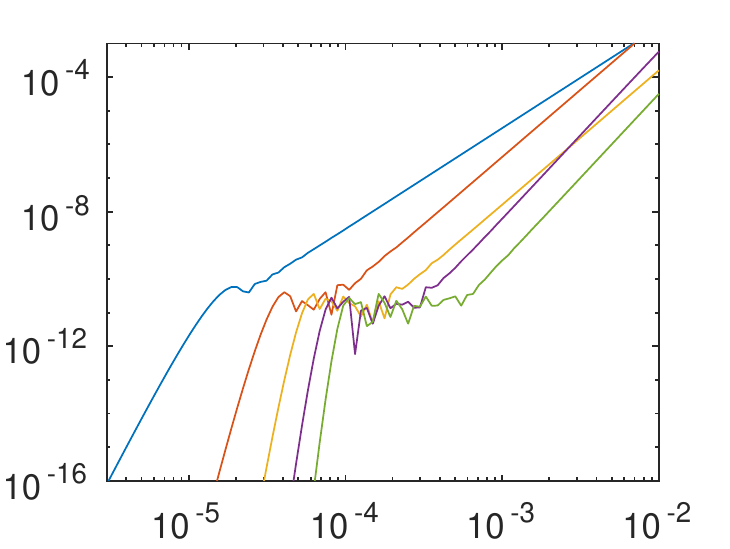}}
\put(127,1){\small $t$}
\put(96,21){\small $\delta_m(t)$}
\put(118,21){\rotatebox{30}{\tiny $10$}}
\put(123,21){\rotatebox{35}{\tiny $20$}}
\put(127,21){\rotatebox{35}{\tiny $30$}}
\put(130.7,21){\rotatebox{40}{\tiny $40$}}
\put(133,21){\rotatebox{40}{\tiny $50$}}
\put(103,39){\includegraphics[width=0.29\textwidth]{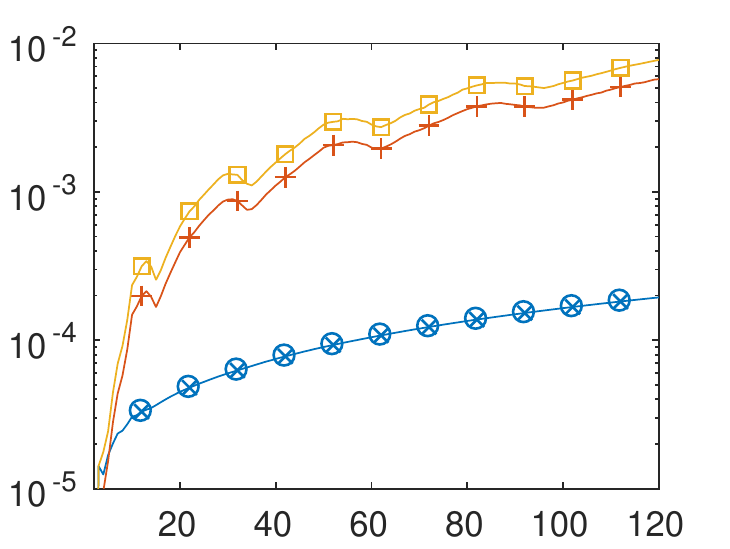}}
\put(127,37){\small $m$}
\put(96,55){\small $t(m)$}
\end{overpic}
\end{tabular}
\caption{Results for the free Schr{\"o}dinger problem with a double well potential.
This Figure shows the time $t(m)$ (top right), which is the smallest $t$ so that $ \zeta_{m}(t) = t\cdot \tol $
for $\tol=10^{-8}$, the true error per unit step (left) $\|l_{m}(t(m))\|_2/t(m)$
and the defect $\delta_m(t)$ (bottom right) for $m\in\{10,20,30,40,50\}$.
The results for $t(m)$ and $\|l_{m}(t(m))\|_2/t(m)$ are given for $\zeta_{m}$
being the upper norm bound given in Theorem~\ref{thm.upperbounderrorfull} ('$\times$'),
Corollary~\ref{thm.upperbounderrorfull2} ('$\circ$'), the generalized residual estimate given in Remark~\ref{thm:generalizedresidualest} ('$+$')
and the effective order estimate given in Remark~\ref{thm:effectiveorderest} ('$\Box$').
The results for Theorem~\ref{thm.upperbounderrorfull} ('$\times$') and
Corollary~\ref{thm.upperbounderrorfull2} ('$\circ$') coincidence in the skew-Hermitian case.
}
\label{fig:numexfS}
\end{figure}
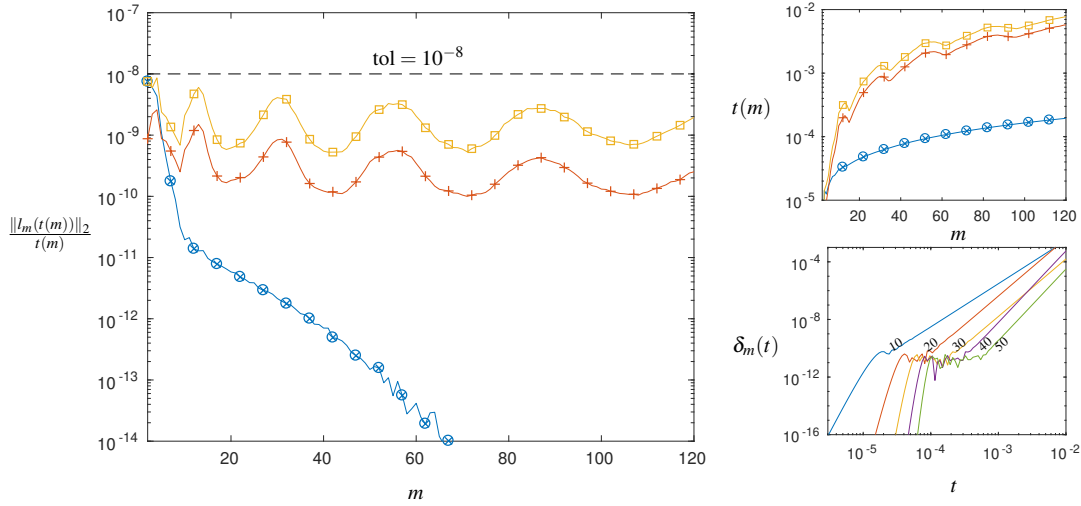
\section{Conclusions and outlook}

In this work various a~posteriori bounds and estimates on the error norm,
which have their origin in an integral representation
of the error using the defect (residual), are studied.
We have characterized the accuracy of these error bounds
by the positioning of Ritz values (i.e., eigenvalues of $H_m$)
on the complex plane.
The case of real Ritz values is the most favorable one
to obtain a tight error bound via an integral on the defect norm
(Corollary~\ref{cor.uppererrorboundreal}).
A new error bound (Theorem~\ref{thm.upperbounderrorfull}) has shown to be tight
if Ritz values are close to the real axis
and in this case favorably compares with existing error bounds.
We further recapitulate an existing error bound
(Corollary~\ref{thm.upperbounderrorfull2})
which remains relevant, especially for the case of
Ritz values with a significant imaginary part.
In addition for the error bound
in Theorem~\ref{thm.upperbounderrorfull} and Corollary~\ref{thm.upperbounderrorfull2},
we have provided a criterion to quantify the achieved accuracy on the run.
For an illustration of the claims concerning the new error bound
we primary refer to the numerical example
given in Subsection~\ref{sec.numexp.cv}.
The quadrature-based error estimates in Subsection~\ref{sec.quadest}
(e.g. the generalized residual estimate)
do not yield proven upper bounds on the error norm
and we adressed special cases (e.g. the numerical example in Subsection~\ref{sec.numexp.fS})
for which the reliability of these estimates can be problematic.
These cases are also analyzed in terms of Ritz values in Subsection~\ref{sec.quadest}
and this relation can be of further interested for a numerical implementation.
Nevertheless, in most cases the quadrature-based estimates remain valid,
whereat the effective order quadrature stands out in terms of performance.

We also remark that the theory provided in our work
gives the possibility to adapt the choice of the error estimate
on the fly to obtain an estimate which is as reliable, accurate and
economic as possible.
A numerical implementation is the topic of further work.

\appendix

\section{Properties of the Krylov subspace in exact and floating point arithmetic}\label{sec.appendixAKrylov}
Let $H_m = V_m^\ast A V_m$ and $V_m^\ast V_m =I_{m\times m}$ in exact arithmetic.
For $z\in\NR(H_m)$ (numerical range of $H_m$) there exists $x\in\C^m$ with
\begin{equation}\label{eq.NRHmissubsetAexactar}
z=\frac{x^\ast H_m x }{x^\ast x} = \frac{x^\ast V_m^\ast A V_m x }{x^\ast V_m^\ast V_m x} = \frac{y^\ast  A y }{y^\ast y},~~~\text{for $y=V_m x$},
\end{equation}
whence $\NR(H_m)\subseteq \NR(A)$.

Similar results hold in floating point arithmetic with relative machine precision~$\varepsilon$
and certain additional assumptions.
Assume there exists an orthonormal basis $\widehat{V}_m\in\C^{n\times m}$
and a perturbation $\widetilde{U}_m\in\C^{n\times m}$,
which is sufficiently small in norm
(i.e., there exists a moderate constant $C_3$ with $\|\widetilde{U}_m\|_2 \leq C_3 \varepsilon$), with
\begin{equation}\label{eq.Hmfloatinga}
H_m = \widehat{V}_m^\ast A \widehat{V}_m + \widetilde{U}_m.
\end{equation}
With assumption~\eqref{eq.Hmfloatinga} and basic properties of the numerical range we obtain
\begin{equation}\label{eq.Hmnrange1}
\NR(H_m) \subseteq \NR(\widehat{V}_m^\ast A \widehat{V}_m) + \NR(\widetilde{U}_m).
\end{equation}
Similar to~\eqref{eq.NRHmissubsetAexactar} we obtain
\begin{equation}\label{eq.Hmnrange2}
\NR(\widehat{V}_m^\ast A \widehat{V}_m) \subseteq \NR(A).
\end{equation}
Then we combine~\eqref{eq.Hmnrange1} and~\eqref{eq.Hmnrange2} and make use of $\|\widetilde{U}_m\|_2 \leq C_3 \varepsilon$ to obtain
$$
\NR(H_m) \subseteq U_{C_3 \varepsilon}(\NR(A)),
$$
with $U_{C_3\varepsilon}(\NR(A))$ being the neighborhood of $\NR(A)$ with a distance $C_3\varepsilon$.

In~\cite[Theorem 5]{Si84} the existence of the representation~\eqref{eq.Hmfloatinga}
is proven for the Lanczos method with a sufficiently small constant $C_3$
and the assumption that the Krylov basis is semiorthogonal.

For the general case of the Arnoldi method the representation~\eqref{eq.Hmfloatinga}
can be derived using~\eqref{eq.Krylovidcomputer},~\eqref{eq.Krylovroundoffconst1}
and an additional condition on the level of orthogonality of the Krylov basis,
e.g., assuming that an orthonormal basis $\widehat{V}_m$ exists for which
$\|\widehat{V}_m - V_m\|_2$ is small enough
(see also~\cite[Theorem 2.1]{BLR00} and references therein).


%
%
%
\section{Some properties of divided differences}\label{sec.propertiesofdd}
%
%
%
\begin{proof}[of \,{\bf Proposition~\ref{thm.intphiA}}]
For $p \in \N_0$ and any $A\in\C^{m\times m}$, $w\in\C^m$,
from the series representation~\eqref{eq.defvarphip} we obtain
\begin{equation}\label{eq.intphiptoppp1}
\int_0^t s^p \varphi_p(sA)w \,\dd s
= \int_0^t\Big( \sum_{k=0}^{\infty} \frac{s^{k+p} A^k w}{(k+p)!}\Big) \,\dd s
= \sum_{k=0}^{\infty} \frac{t^{k+p+1} A^k w}{(k+p+1)!}
= t^{p+1} \varphi_{p+1}(tA)w.
\end{equation}
This identity carries over to divided differences in the following way. Let
\begin{equation*}
\Theta_m =
\begin{pmatrix}
\lambda_1&  &&\\
1 & \lambda_2 &&\\
  & \ddots &\ddots& \\
 && 1&\lambda_m
\end{pmatrix}\in\C^{m\times m}.
\end{equation*}
As a consequence of the Opitz formula, see~\cite{Op64} and remarks in~\cite[Proposition 25]{Bo05}, we have
\begin{equation}\label{eq.opitzddtomatfct}
(\varphi_p)_t[\lambda_1,\ldots,\lambda_m] =
e_m^\ast \varphi_p(t \Theta_m)e_1.
\end{equation}
Using~\eqref{eq.intphiptoppp1} and~\eqref{eq.opitzddtomatfct} we obtain
\begin{equation*}
\int_0^t s^p (\varphi_p)_s[\lambda_1,\ldots,\lambda_m] \,\dd s
= e_m^\ast \int_0^t s^p \varphi_p(s\Theta_m) e_1\,\dd s
= e_m^\ast  t^{p+1} \varphi_{p+1}(t\Theta_m)e_1
= t^{p+1} (\varphi_{p+1})_t[\lambda_1,\ldots,\lambda_m],
\end{equation*}
which completes the proof.
\qed
\end{proof}

%
%
%

\begin{remark}
We will make use of the following integral representation for divided differences,
the so-called Hermite-Genocchi formula,~\cite[eq. (B.25)]{Hi08}.
With the differential operator $(D^{(m-1)}f_t)(\lambda)  = \frac{\dd^{m-1}}{\dd \lambda^{m-1}}f(t\lambda)$,
\begin{equation} \label{eq.ddhermitegenocchi}
\begin{aligned}
{f_t}[\lambda_1,\ldots,\lambda_m]
&= \int_{[\lambda_1,\ldots,\lambda_m]}D^{(m-1)}f_t \\
& = \int_0^1
    \int_0^{s_1} \cdots
    \int_0^{s_{m-2}}
    D^{(m-1)}f \Big(\lambda_1 + \sum_{j=1}^{m-1}s_j(\lambda_{j+1}-\lambda_{j}) \Big)
    \dd s_{m-1}\,\ldots\,\dd s_{2}\,\dd s_1.
\end{aligned}
\end{equation}
\end{remark}

\begin{proof}[of \,{\bf Proposition~\ref{thm.realnodesposphiandupperbound}}]
Applying~\eqref{eq.ddhermitegenocchi} to the exponential function gives
\begin{align*}
|\exp_t[\lambda_1,\ldots,\lambda_k]|
& \leq \int_0^1
    \int_0^{s_1} \cdots
    \int_0^{s_{k-2}}\,
    t^{k-1} \Big|\exp\Big(\lambda_1 + \sum_{j=1}^{k-1}s_j(\lambda_{j+1}-\lambda_{j}) \Big)\Big|
    \,\dd s_{k-1}\,\ldots\,\dd s_{2}\,\dd s_1\\
& = \int_0^1
    \int_0^{s_1} \cdots
    \int_0^{s_{k-2}}\,
    t^{k-1} \exp\Big(\xi_1 + \sum_{j=1}^{k-1}s_j(\xi_{j+1}-\xi_{j}) \Big)
    \dd s_{k-1}\,\ldots\,\dd s_{2}\,\dd s_1\\
& = \exp_t[\xi_1,\ldots,\xi_k],
\end{align*}
which completes the proof.
\qed
\end{proof}

\begin{proof}[of \,{\bf Proposition~\ref{thm.ddlamlowerbound}}]
We use~\eqref{eq.ddhermitegenocchi} to obtain
\begin{align*}
\exp_t[\lambda_1,\ldots,\lambda_k]
& = \int_0^1
    \int_0^{s_1} \cdots
    \int_0^{s_{k-2}}\,
    t^{k-1} \exp\Big(t\Big(\lambda_1 + \sum_{j=1}^{k-1}s_j(\lambda_{j+1}-\lambda_{j}) \Big)\Big)
    \dd s_{k-1}\,\ldots\,\dd s_{2}\,\dd s_1\\
& = \int_0^1
    \int_0^{s_1} \cdots
    \int_0^{s_{k-2}}\,
    t^{k-1} \\
   &~~~~~~\cdot\Big[\cos\Big(t\Big(\eta_1 + \sum_{j=1}^{k-1}s_j(\eta_{j+1}-\eta_{j}) \Big)\Big)
+\ii\sin\Big(t\Big(\eta_1 + \sum_{j=1}^{k-1}s_j(\eta_{j+1}-\eta_{j}) \Big)\Big)\Big]\\
   &~~~~~~ \cdot\,\exp\Big(t\Big(\xi_1 + \sum_{j=1}^{k-1}s_j(\xi_{j+1}-\xi_{j}) \Big)\Big)
    \;\dd s_{k-1}\,\ldots\,\dd s_{2}\,\dd s_1\\
&=(\cos(tx)+\ii\sin(ty))\exp_t[\xi_1,\ldots,\xi_k]
\quad \text{for certain $ x,y \in \text{Conv}(\{\eta_1,\ldots,\eta_k\}). $}
\end{align*}
Here, in the last step we have used the Mean Value Theorem for the integral.
In this way we end up with the estimate
$$
|\exp_t[\lambda_1,\ldots,\lambda_m]| = |\cos(tx)+\ii\sin(ty)|\cdot\exp_t[\xi_1,\ldots,\xi_m].
$$
With $|tx|, |ty| \leq \widetilde{\eta}_t  < \pi/2 $ we obtain
$$
\cos(\widetilde{\eta}_t) \leq \cos(tx) \leq |\cos(tx)+\ii\sin(ty)|,
$$
which completes the proof.
\qed
\end{proof}

\section{A new asymptotic expansion of divided differences}\label{sec.expansionofrho}

Our goal is to derive an asymptotic expansion for $ |\exp_t[\lambda_1,\ldots,\lambda_m]| $,
see Theorem~\ref{thm.rho12asymptotic} at the end of this section.

Let $\lambda_1,\ldots,\lambda_m \in\C$.
We use the shortcut $\kappa_k$ for the divided differences of power functions,
\begin{equation}\label{eq.defkappa}
\kappa_k=(\cdot)^{m-1+k}[\lambda_1,\ldots,\lambda_m]~~~\text{for}~~k\in\N_0,
\end{equation}
where $(\cdot)^{j}: z\mapsto z^{j}$ for $j\in\N_0$. Note that
\begin{equation*}
(\cdot)^{j}[\lambda_1,\ldots,\lambda_m] = 0 ~~~\text{for}~~j=0,\ldots,m-2.
\end{equation*}
With the notation~\eqref{eq.defkappa} and the series representation of the exponential function we obtain
\begin{subequations}
\begin{align}
\exp_t[\lambda_1,\ldots,\lambda_m]
=& \sum_{j=0}^{\infty} \frac{t^j\,(\cdot)^{j}[\lambda_1,\ldots,\lambda_m]}{j!}
= t^{m-1} \sum_{k=0}^{\infty} \frac{t^k \kappa_k }{(m-1+k)!}\label{eq.expddseries}\\
=& \frac{ t^{m-1}}{(m-1)!} + \Ono(t^m)~~~\text{for~~$t\to 0$.}\label{eq.expddseriestt0}
\end{align}
\end{subequations}
We also introduce the notation
\begin{equation}\label{eq.sl}
S_l = \sum_{j=1}^m \lambda_j^l,~~~l\in\N.
\end{equation}
For $\kappa_0$, $\kappa_1$ and $\kappa_2$ we obtain the following formula.
\begin{proposition}\label{propo:emHke1}
For $\kappa_k$ introduced in~\eqref{eq.defkappa} we have
$$
\kappa_0 = 1,~~~\kappa_1 = S_1,~~~ \kappa_2 = (S_1^2+S_2)/2.
$$
\end{proposition}
\begin{proof}
This follows from~\cite[eq. (27)]{Bo05}.
\qed
\end{proof}

To simplify the notation we write
$$
f(t) = |\exp_t[\lambda_1,\ldots,\lambda_m]|.
$$
The following asymptotic expansion of $f(t)$ for $ t \to 0 $ is motivated by the concept of effective order.
The effective order of the function $ f(t) $ can be understood as the slope of the double-logarithmic function
\begin{equation*}
\ln(f(\ee^{\tau}))~~~\text{with}~~\tau=\ln t,
\quad \text{and with derivative}~~\frac{f'(\ee^{\tau})\,\ee^{\tau}}{f(\ee^{\tau})}.
\end{equation*}
We denote the effective order by
\begin{subequations}
\begin{align}
\rho(t) &= \frac{f'(t)\,t}{f(t)}, \label{def:effectiveorder} \\
\text{satisfying} \quad \rho(t)/t&=\big(\log(f(t))\big)'. \label{eq.rhoovertlogftd}
\end{align}
\end{subequations}

We now analyze the divided differences close to an asymptotic regime
under the assumption $f(t)>0$, which holds for sufficiently small $t>0$.
The effective order $\rho(t)$ is then well-defined by~\eqref{def:effectiveorder}.
The following expansion~\eqref{eq.rhoansatz} for $\rho(t)$
is be considered in an asymptotic sense for $ t \to 0 $;
convergence of the series is not an issue here.

We make the ansatz
\begin{equation}\label{eq.rhoansatz}
\rho(t)=\sum_{k=0}^\infty \rho_k t^k
\end{equation}
Using~\eqref{eq.rhoansatz} in~\eqref{eq.rhoovertlogftd} we obtain
\begin{align*}
\frac{\rho(t)}{t} =\Big( \rho_0 \log(t) + \sum_{k=1}^\infty \rho_k t^{k}/k\Big)' &= (\log(f(t)))'\\
c\,\exp\Big(\rho_0 \log(t) + \sum_{k=1}^\infty \rho_k t^{k}/k \Big) &= f(t),\\
c\,t^{\rho_0} \exp\Big(\sum_{k=1}^\infty \rho_k t^{k}/k \Big) &= f(t).
\end{align*}
From~\eqref{eq.expddseriestt0} we see that $c=1/(m-1)!$ and $ \rho_0=m-1$, whence
\begin{equation}\label{eq.rhoasymansatznoconst}
\rho(t)=m-1+\sum_{k=1}^\infty \rho_k t^k,
\end{equation}
and for sufficiently small $t$,
\begin{equation}\label{eq.defectnormbyrho}
f(t) = |\exp_t[\lambda_1,\ldots,\lambda_m]| = \frac{t^{m-1}}{(m-1)!}\exp\Big(\sum_{k=1}^\infty \rho_k t^{k}/k \Big).
\end{equation}

We aim for deriving a formula for the coefficients $\rho_k$.
To avoid the square roots we choose $q(t) = f(t)^2 $, such that $ f'(t) = q'(t)/(2q(t)^{1/2})$.
Due to~\eqref{def:effectiveorder} the effective order $\rho(t)$ satisfies
\begin{equation}\label{eq.rhoimplicit}
q(t) \rho(t)=q'(t)t/2.
\end{equation}
We proceed by rewriting $q(t)$ and $q'(t)$ to obtain a formulation for $\rho_k$ ($k\geq 1$) via~\eqref{eq.rhoimplicit}.
From~\eqref{eq.expddseries},
$$
q(t) = |\exp_t[\lambda_1,\ldots,\lambda_m]|^2
= t^{2(m-1)} \Big(\sum_{k=0}^\infty\frac{t^k\kappa_k}{(m-1+k)!}\Big)
             \Big(\sum_{\ell=0}^\infty\frac{t^\ell{\overline{\kappa}}_\ell}{(m-1+\ell)!}\Big).
$$
The representation of $ q(t) $ as well as $ t q'(t)/2 $ as a Cauchy product can be written in the form
\begin{equation}\label{eq.qqprime}
q(t) = \frac{t^{2(m-1)}}{((m-1)!)^2} \sum_{k=0}^\infty \alpha_k t^k,~~~\text{and}~~~
tq'(t)/2 = \frac{t^{2(m-1)}}{((m-1)!)^2} \sum_{k=0}^\infty \big((m-1)+k/2\big) \alpha_k t^k,
\end{equation}
with coefficients $\alpha_k$ given by
$$
\alpha_0=1,~~~\text{and}~~~\alpha_k
= \sum_{j=0}^k  \frac{((m-1)!)^2\,\kappa_j{\overline{\kappa}}_{k-j} }{(m-1+j)!\,(m-1+k-j)!}
~~~\text{for~\,$k\in\N$}.
$$
With $\kappa_0=1$ (see Proposition~\ref{propo:emHke1}) this can be written as
\begin{equation}\label{eq.defalpha}
\alpha_k = \frac{2 (m-1)!\,\real(\kappa_k)}{(m-1+k)!}
+ \sum_{j=1}^{k-1}  \frac{((m-1)!)^2\,\kappa_j{\overline{\kappa}}_{k-j}}{(m-1+j)!\,(m-1+k-j)!}
~~~\text{for~\,$k\in\N$}.
\end{equation}
Furthermore, from~\eqref{eq.rhoasymansatznoconst} and~\eqref{eq.qqprime}
we obtain a representation of $ q(t) \rho(t) $ in form of a Cauchy product,
\begin{equation}\label{eq.qrho}
q(t) \rho(t) = \frac{t^{2(m-1)}}{((m-1)!)^2} \sum_{k=0}^\infty \theta_k t^k,
~~~\text{with}~~\theta_k = \sum_{j=0}^{k-1} \alpha_j \rho_{k-j}+ (m-1)\alpha_k ,~~~k\in\N_0.
\end{equation}
We remark that~\eqref{eq.qrho} only holds for $ t $ small enough.
With $\alpha_0=1$, in~\eqref{eq.qrho} we have
\begin{equation}\label{eq.rhotheta2}
\theta_0= m-1,~~~\text{and}~~~
\theta_k = \rho_k + \sum_{j=1}^{k-1} \alpha_j \rho_{k-j} + (m-1)\alpha_k ,~~~k\in\N.
\end{equation}
For the implicit equation~\eqref{eq.rhoimplicit} we combine~\eqref{eq.qqprime} and~\eqref{eq.qrho} to obtain
\begin{equation}\label{eq.rhoimplseries}
\sum_{k=0}^\infty \theta_k t^k = \sum_{k=0}^\infty (m-1+k/2) \alpha_k t^k.
\end{equation}
Comparing coefficients of $t^k$ in~\eqref{eq.rhoimplseries} and using~\eqref{eq.rhotheta2} we conclude
\begin{equation}\label{eq.recursiveformularhok}
\theta_k = (m-1+k/2) \alpha_k,~~~\text{and}
~~~\rho_k = \frac{k\alpha_k }{2} - \sum_{l=1}^{k} \alpha_l \rho_{k-l},~~~k \geq 1.
\end{equation}
From~\eqref{eq.recursiveformularhok} we obtain a recursion for the coefficients $\rho_k$
in the expansion~\eqref{eq.rhoasymansatznoconst}
which can be resolved using~\eqref{eq.defkappa} and~\eqref{eq.defalpha}.

We now evaluate the lower coefficients of $\rho(t)$.
For $\alpha_1$ and $\alpha_2$, using Proposition~\ref{propo:emHke1} in~\eqref{eq.defalpha} gives
\begin{equation}\label{eq.alphaS1S2}
\alpha_1 
= \frac{2\,\real(\kappa_1)}{m} =\frac{2\,\real(S_1)}{m},
~~~\text{and}~~~\alpha_2 
= \frac{|\kappa_1|^2}{m^2} + \frac{2\,\real(\kappa_2) }{m(m+1)}
=\frac{|S_1|^2}{m^2} + \frac{\real(S_1^2+S_2)}{m(m+1)},
\end{equation}
with $ S_1 $, $ S_2 $ according to definition~\eqref{eq.sl}
From the recursion in~\eqref{eq.recursiveformularhok} we have
\begin{equation}\label{def.rho12}
\rho_1 = \frac{\alpha_1}{2},~~~\rho_2= \frac{1}{2}\big( 2 \alpha_2 - \alpha_1^2 \big),
\end{equation}
and combining~\eqref{eq.alphaS1S2} with~\eqref{def.rho12} we eventually obtain
\begin{equation}\label{eq.rho12}
\begin{aligned}
\rho_1 &= \frac{\real(S_1)}{m}, \\
\rho_2 &= \frac{|S_1|^2}{m^2} + \frac{\real(S_1^2+S_2) }{m(m+1)} - \frac{2\,\real(S_1)^2}{m^2}
 = \frac{\imag(S_1)^2- \real(S_1)^2}{m^2}  + \frac{\real(S_1^2+S_2) }{m(m+1)}. 
\end{aligned}
\end{equation}
To study the influence of the real and imaginary parts
of the nodes $\lambda_j=\xi_j + \ii \eta_j$ we introduce the notation
\begin{equation}\label{eq.defslk}
S_{lk} = \sum_{j=1}^m\xi_j^l \eta_j^k,~~~l,k\in\N_0.
\end{equation}
Basic computations, mostly binomial sums in~\eqref{eq.sl}, show
$$
S_1 = S_{10} + \ii S_{01},~~~S_2 = S_{20} + 2\ii S_{11}  - S_{02},~~~\text{and}~~S_1^2 = S_{10}^2 + \ii S_{10}S_{01} - S_{01}^2,
$$
and
\begin{equation}\label{def.RealImagS1S2S12x}
\imag(S_1)=S_{01},~~~\real(S_1)=S_{10},~~~\real(S_2)=S_{20}-S_{02},~~~\text{and}~~\real(S_1^2)=S_{10}^2-S_{01}^2.
\end{equation}
Combining~\eqref{eq.rho12} with~\eqref{def.RealImagS1S2S12x} gives
\begin{equation}\label{eq.rho12withSlk}
\rho_1 =\frac{S_{10}}{m},~~~\text{and}~~
\rho_2 = \frac{ S_{01}^2 - S_{10}^2}{m^2(m+1)} + \frac{S_{20}-S_{02}}{m(m+1)}.
\end{equation}

After all these technicalities we arrive at the following asymptotic expansion.
\begin{theorem}\label{thm.rho12asymptotic}
Assume that for $\lambda_j=\xi_j + \ii \eta_j$ at least
one of the sequences $\{\xi_j\}_{j=1}^m$ and $\{\eta_j\}_{j=1}^m$ is not constant,
and $\xi_j\leq 0$ for $j=1,\ldots,m$.
Let $ \avg(\xi) = \sum_{j= 1}^m \xi_j/m$ be the average and
$ \var(\xi) = \sum_{j=1}^m (\xi_j-\avg(\xi))^2/m $
be the variance
of $\{\xi_1,\ldots,\xi_m\}$, and $ \var(\eta) $ the variance of $\{\eta_1,\ldots,\eta_m\}$.
Then,
\begin{enumerate}[(a)]
\item
$$
|\exp_t[\lambda_1,\ldots,\lambda_m]| = \frac{t^{m-1}}{(m-1)!}\exp\big(\rho_1 t + \rho_2t^2/2  + \Ono(t^3) \big)
~~~\text{for~\,$t\to 0$},
$$
with
$$
\rho_1 = \avg(\xi),~~~ \rho_2 = \frac{\var(\xi) - \var(\eta)}{m+1},
$$
and either $\rho_1\neq 0$ or $\rho_2\neq 0$.
\item
The derivative of the effective order $\rho(t)$ (see~\eqref{def:effectiveorder}) satisfies
$\rho'(t) = \rho_1 + \rho_2t + \Ono(t^2)$ for \,$t\to 0$, and
$$
\rho'(0+) < 0.
$$
\end{enumerate}
\end{theorem}
\begin{proof}
We use the expansion~\eqref{eq.defectnormbyrho} for sufficiently small $t$.
For the variance we obtain
$$
\var(\xi) = \frac{1}{m} \sum_{j=1}^m (\xi_j-\avg(\xi))^2 = \frac{1}{m} \Big(\sum_{j=1}^m \xi_j^2 - \frac{1}{m}\big(\sum_{j=1}^m \xi_j\big)^2\Big).
$$
The first coefficients $\rho_1$ and $\rho_2$ are given in~\eqref{eq.rho12withSlk}.
With the notation from~\eqref{eq.defslk} we observe
$\avg(\xi) = S_{10}/m $ (for the average $\avg(\xi)$) and $\var(\xi) = (S_{20}-S_{10}^2/m)/m $,
$\var(\eta) = (S_{02}-S_{01}^2/m)/m $ (for the variance $\var(\xi)$ and $\var(\eta)$, respectively),
whence
$$
\rho_1 = \avg(\xi),~~~\text{and}~~ \rho_2 = \frac{\var(\xi) - \var(\eta)}{m+1}.
$$
With $\xi_1,\ldots,\xi_m\leq 0$ for $j=1,\ldots,m$ we obtain $\rho_1 \leq 0$
and $\rho_1 = 0$ iff $\xi_1,\ldots,\xi_m = 0$.
For the case $\xi_1,\ldots,\xi_m = 0$ we obtain $\var(\xi) = 0 $ and
$$
\rho_2 = - \frac{\var(\eta)}{m+1} \leq 0.
$$
Here, $\rho_2=0$ only in the trivial case with $\xi_1,\ldots,\xi_m = 0$ and a constant sequence $\eta_1,\ldots,\eta_m$.
This proves~(a).
For the proof of~(b) we take the derivative of $\rho(t)$
in an asymptotic sense and make use of $\rho_1\leq 0$ and $\rho_2<0$ iff $\rho_1=0$,
see~(a).
\qed
\end{proof}

\begin{acknowledgements}
This work was supported by the Doctoral College TU-D,
Technische Universit{\"a}t Wien.
\end{acknowledgements}


\end{document}